\documentclass[DIV=14,letterpaper]{scrartcl}

\usepackage{tikz}
\usetikzlibrary{calc}
\usetikzlibrary{matrix}
\usepackage{amsmath,amssymb,amsfonts,amsthm}
\usepackage{graphicx}
\allowdisplaybreaks
\usepackage{subfigure,multicol,multirow}
\usepackage{makecell}
\usepackage{hyperref} 
\usepackage{diagbox}

\usepackage{bbm}

\theoremstyle{plain}

\newtheorem{pro}{\hspace{6mm}Proposition}[section]
\newtheorem{lem}{\hspace{6mm}Lemma}[section]

\theoremstyle{definition}

\theoremstyle{remark}

\newcommand{\V}[1]{\mathbf{#1}}

\newcommand{\email}[1]{\href{mailto:#1}{#1}}

\title{Efficient parameter-robust preconditioners for linear poroelasticity and elasticity in the primal formulation%
}
\author{
Weizhang Huang\thanks{Department of Mathematics, the University of Kansas, 1460 Jayhawk Blvd, Lawrence, KS 66045, USA (\email{whuang@ku.edu}).}
\and
Zhuoran Wang\thanks{Department of Mathematics, the University of Kansas, 1460 Jayhawk Blvd, Lawrence, KS 66045, USA (\email{wangzr@ku.edu}).}
}
\date{} 

\begin{document}

\maketitle

\begin{abstract}
Poroelasticity problems play an important role in various engineering, geophysical, and biological applications.
Their full discretization results in a large-scale saddle-point system at each time step that is becoming
singular for locking cases and needs effective preconditioners for its fast iterative solution.
Instead of constructing spectrally equivalent ones,
we develop nonsingular preconditioners so that the eigenvalues of the preconditioned system consist of a cluster around
$1$ and an outlier in the order of $1/\lambda$, where $\lambda$ is a Lam\'{e} constant that is large for locking cases.
It is known that the convergence factor of GMRES is bounded by the radius of the cluster for this type of systems.
Both two- and three-field block triangular Schur complement preconditioners are studied.
Upper bounds of the radius of the eigenvalue cluster for those systems
are obtained and shown to be related to the inf-sup condition but independent
of mesh size, time step, and locking parameters, which reflects the robustness of the preconditioners
with respect to parameter variations. 
Moreover, the developed preconditioners do not need to compute the Schur complement
and neither require exact inversion of diagonal blocks except the leading one.
A locking-free weak Galerkin finite element method and the implicit Euler scheme are used for the discretization
of the governing equation. Both two- and three-dimensional numerical results are presented to confirm the effectiveness
and parameter-robustness of the developed preconditioners.
\end{abstract}



\section{Introduction}
We consider Biot's model \cite{Biot1941}  of linear poroelasticity that is governed by
\begin{equation}
\begin{cases}
  \displaystyle
  -\nabla \cdot
  \big ( 2\mu \varepsilon(\mathbf{u})
    + \lambda (\nabla \cdot \mathbf{u}) \mathbf{I}
  \big )
  + \alpha \nabla p
  = \mathbf{f}, \quad \text{ in } \Omega
\\
  \displaystyle
  \partial_t
  \big ( \alpha \nabla \cdot \mathbf{u} +  c_0 p \big)
    + \nabla \cdot \left( -\mathbf{K} \nabla p \right)
  = s, \quad \text{ in } \Omega
\end{cases}
\label{EqnPoroElas}
\end{equation}
where
$\Omega \in \mathbb{R}^d \; (d \geq 2)$ is a bounded Lipschitz domain,
$ \mathbf{u} $ is the solid displacement,
$ \varepsilon(\mathbf{u}) = \frac{1}{2} ( \nabla \mathbf{u} + (\nabla \mathbf{u})^T ) $
is the strain tensor,
$ \lambda$ and $\mu $ are Lam\'{e} constants,
$ \sigma(\mathbf{u}) = 2\mu \varepsilon(\mathbf{u}) + \lambda (\nabla \cdot \mathbf{u}) \mathbf{I} $
is the Cauchy stress for the solid,
$ \mathbf{I} $ is the identity operator,
$ \mathbf{f} $ is a body force,
$ p $ is the fluid pressure,
$ s $ is the fluid source,
$ \alpha $ (usually close to $1$) is the {Biot-Willis} constant,
$ c_0 \ge 0 $ is the constrained storage capacity,
and $ \mathbf{K} $ is a permeability tensor.
The Lam\'{e} constants are related to the elasticity modulus $E$ and the Poisson's ratio $\nu$ as
\[
\lambda = \frac{\nu E}{(1-2\nu)(1+\nu)},\quad \mu = \frac{E}{2(1+\nu)} .
\]
We consider here $\mathbf{K} = \kappa \mathbf{I}$ and
Dirichlet boundary conditions for both the displacement and pressure, i.e.,
\begin{equation}
\mathbf{u}|_{\partial \Omega} = \mathbf{u}_D, \quad p|_{\partial \Omega} = p_D,
\label{BC-1}
\end{equation}
where $\mathbf{u}_D$ and $p_D$ are given functions.

Poroelasticity problems are important in various engineering, geophysical, and biological applications,
such as in the development of the robust mechanical modeling of tissues and cells \cite{Malandrino_2019}
and the research on the poroelastic component of the postseismic process \cite{McCormack4_Geo_2020}.
{Locking, which occurs when $\lambda \to \infty$ or $\nu \to 0.5$ 
or when $c_0 = 0$ and $\kappa \to 0$, is one of the challenges in developing appropriate numerical methods for poroelasticity problems.
The first situation corresponds to the case when the solid material is nearly incompressible
while the second situation corresponds to the case when the system becomes incompressible
due to low permeability. When locking occurs, the convergence order of the discretization
can deteriorate, and the resulting algebraic system becomes nearly singular.}
{Locking-free numerical methods that have been developed for} poroelasticity problems include 
mixed finite element methods \cite{AmbartKhatYotov_CMAME_2020},
virtual element methods \cite{Burger4_AdvComputMath_2021,Coulet4_ComputGeosci_2020},
discontinuous Galerkin methods \cite{RIVIERE2017666},
enriched Galerkin methods \cite{KADEETHUM2021110030,LeeYi_JSC_2023},
finite volume methods \cite{TEREKHOV2022111225},
and weak Galerkin (WG) finite element methods \cite{Wang2TavLiu_JCAM_2024}.
These numerical methods typically lead to a large-scale linear system at each time step,
and it is crucial that the system is solved efficiently.

A common strategy for the efficient solution of large-scale linear systems is to develop effective and efficient preconditioners
for iterative methods.
There are several recent works in this direction for poroelastic models. 
For example,
Lee et al. \cite{Lee-SISC-2017} introduced parameter-robust three-field block diagonal preconditioners
based on stability consideration and the operator preconditioning approach.
Boon et al. \cite{Boon4_SISC_2021} constructed parameter-robust preconditioners for four-field numerical schemes.
Adler et al. \cite{Adler6_SISC_2020} proposed norm-equivalent and field-of-value-equivalent block preconditioners
for the stabilized discretization of poroelasticity problems in a three-field approach.
A general framework was proposed and several preconditioners for poroelasticity problems in two-
and three-field numerical schemes were discussed using the framework by Chen et al. \cite{ChenHongXuYang_CMAME_2020}.
More recently, Rodrigo et al. \cite{Rodrigo6_SeMA_2024} presented preconditioners for two- and three-field numerical schemes. 
It is pointed out that the existing preconditioners share the common feature that they are equivalent to the coefficient matrix of
the underlying system in terms of spectrum, norm, and/or field of value.
{While they are generally effective and parameter-robust, 
those preconditioners are singular or nearly singular
due to their spectral equivalence to the original system.
This makes them more challenging to construct and their inversion more expensive to carry out than
nonsingular preconditioners.
}

In this work, we consider a different strategy with which we do not seek spectrally equivalent preconditioners.
Instead, we look for nonsingular preconditioners that can take care of non-small eigenvalues while keeping small ones.
With this approach, we have more flexibility to construct preconditioners. Moreover, these preconditioners
are more straightforward to construct and the action of their inversion is less expensive
to carry out.
Analytical analysis and numerical experiments will show that the developed preconditioners are effective and robust
with respect to parameter variations for both linear poroelasticity and elasticity problems.
More specifically, we employ the two-field discretization by a weak Galerkin finite element method coupled
with the implicit Euler scheme. The method has been shown in \cite{Wang2TavLiu_JCAM_2024}
to be free of locking and have optimal-order convergence in pressure, displacement, Darcy velocity, stress, and dilation.
The resulting discrete system is a saddle-point system of two-by-two blocks. Iterative solution and preconditioning
for general saddle-point systems have been
studied extensively; e.g., see \cite{BenziGolubLiesen-2005,Benzi2008} and references therein.
However, for the current saddle-point system, the leading block, i.e., the (1,1) block,
is becoming singular as $\lambda \to \infty$, the (1,2) and (2,1) blocks are singular, and
the (2,2) block is nonzero while becoming singular as $\Delta t \to 0$.
These features make it difficult to apply the existing theory of preconditioning for general saddle-point systems
to the current system. Here, we look for nonsingular preconditioners that
can take care of non-small eigenvalues while keeping small ones. This allows us to explore and use bounds of the Schur complement
and develop preconditioners that do not need to compute the Schur complement.
Two- and three-field block upper triangular Schur complement preconditioners are studied.
{For the two-field forumlation, the poroelasticity problem is solved with a nested iteration
where the outer loop is for the poroelasticity problem and
the inner loop is for its leading block that corresponds to a linear elasticity problem.
A spectrally equivalent, straightforward preconditioner is constructed for the outer loop.
For the inner loop (or for an elasticity problem), a nonsingular, spectrally non-equivalent
preconditioner is used and the eigenvalues of the corresponding preconditioned system
are shown to consist of a cluster of eigenvalues around $1$ and an outlier in the order of $1/\lambda$.
The situation for the three-field formulation is similar to that for the elasticity problem, i.e.,
a nonsingular, spectrally non-equivalent
preconditioner is used and the eigenvalues of the corresponding preconditioned system
contain an outlier in the order of $1/\lambda$ and a cluster around 1.
}
It has been shown by Campbell et al. \cite{Campbell-1996} that the convergence factor of
the generalized minimal residual method (GMRES) \cite{GMRES-1986} is bounded by the radius of the eigenvalue cluster
for this type of systems. For our current situation, upper bounds of the radius are obtained and shown
to be related to the inf-sup condition but independent
of mesh size, time step, and locking parameters, which reflects the robustness of the preconditioners
with respect to parameter variations. 
Numerical results in two and three dimensions are presented to confirm the effectiveness
and parameter-robustness of the developed preconditioners. 

The rest of this paper is organized as follows. 
In Section~\ref{SEC:formulation}, we introduce the variational formulation for (\ref{EqnPoroElas}) and its discretization
by the WG method in space and the implicit Euler scheme in time.
We also discuss the distinct features of the resulting linear system.
In Section~\ref{SEC:2by2}, we study the two-field block triangular preconditioners which are spectrally equivalent to the Schur complement for the two-field poroelasticity problems. The inversion of their leading block is equivalent to solving
a linear elasticity problem which is in turn transformed into a saddle-point problem. Preconditioning for the elasticity saddle-point system is also discussed in the section. The performances of these preconditioners are confirmed in the section
with two-dimensional numerical examples. 
In Section~\ref{SEC:3by3}, we transform the two-field poroelasticity system into a three-field system using the so-called solid pressure and total pressure and  develop three-field preconditioners. 
Numerical experiments in both two and three dimensions are presented to showcase the effectiveness and  parameter-robustness
of the preconditioners.
Section~\ref{SEC:conclusions} contains conclusions and further comments.
A brief discussion on block upper triangular Schur complement preconditioning and the eigenvalues of the preconditioned matrix
for general saddle-point systems is given in Appendix~\ref{appendix_A}.

\section{Weak Galerkin finite element discretizations for poroelasticity}
\label{SEC:formulation}

In this section we describe the weak formulation of the linear poroelasticity problem (\ref{EqnPoroElas}) and its 
discretization in space by a weak Galerkin finite element method and in time by the implicit Euler method.
For notational simplicity, we consider the two-dimensional case here. The description works for other dimensions without
major modifications.

\subsection{Weak formulation of poroelasticity problems}

The weak formulation of the poroelasticity problem (\ref{EqnPoroElas}) is 
to find $(\mathbf{u}(\cdot, t), p(\cdot, t)) \in (H^1(\Omega))^{  d}\times H^1(\Omega)$, $0< t \le T$, such that
the boundary condition (\ref{BC-1}) holds in a weak sense and
\begin{equation}
\begin{cases}
    \mu \Big( \nabla \mathbf{u}, \nabla\mathbf{v} \Big)
      + (\lambda+\mu) ({\nabla \cdot \mathbf{u}}, {\nabla \cdot \mathbf{v}})
      - \alpha (p, {\nabla \cdot \mathbf{v}})
      = (\mathbf{f}, \mathbf{v}),
      \quad \forall \mathbf{v}\in (H_{0}^1(\Omega))^{ d}
\\
    \displaystyle
    -\alpha (\nabla \cdot \mathbf{u}_t, q)
    - c_0 \left( p_t, q \right)
      - \left( \kappa \nabla p, \nabla q \right)
    \displaystyle
    =  -\left( s, q \right),
    \quad \forall q\in H_{0}^1(\Omega)
\end{cases}
  \label{poro_variationalform}
\end{equation}
where $(\cdot, \cdot)$ denotes the $L^2$ inner product over $\Omega$.
For linear elasticity, the grad-div and strain-div formulations are mathematically equivalent 
when pure Dirichlet conditions are used. 
In this work, we consider the above grad-div formulation as in \cite{Wang2TavLiu_JCAM_2024}.

\subsection{Full discretization via weak Galerkin and implicit Euler methods}

We now consider the discretization of (\ref{poro_variationalform}).
For spatial discretization, we employ a weak Galerkin finite element method of Wang et al. \cite{Wang2TavLiu_JCAM_2024}
where the method has been shown to be locking free,
have optimal-order convergence in pressure, displacement, Darcy velocity, stress, and dilation,
and satisfy the inf-sup condition for linear elasticity.

Let $\mathcal{E}_h $ be a shape regular, quasi-uniform convex quadrilateral mesh of $\Omega$.
For any element $E\in\mathcal{E}_h$, denote the diameter of the circumscribed circle of $E$ by $h_E$.
Denote the element size of $ \mathcal{E}_h $ by $ h = \max_{E\in\mathcal{E}_h}h_E $
and the set of all edges of $\mathcal{E}_h$ by $\Gamma_h$.
The finite element spaces are defined in the following.

The discrete (scalar) weak function space $W_h$ is defined as
\begin{equation}
W_h=\Big \{p_h=\{p^\circ_h, p^\partial_h\}\; :\; p^\circ_h|_E\in P_0(E),\; p^\partial_h|_e \in P_0(e),
\; \forall E\in\mathcal{E}_h,\; e \in \Gamma_h\Big\},     
\label{Wh}
\end{equation}
where $P_0(E)$ and $P_0(e)$ denote the sets of constant polynomials defined on $E$ and $e$, respectively.
It is worth emphasizing that any function of $W_h$ has two sets of degrees of freedom, i.e., those defined in the interiors
of the elements and those defined on the element edges.
We will also need to use the (interior) function space,
\[
\mathcal{P}_0(\mathcal{E}_h) = \Big\{p_h=\{p^\circ_h\}\; :\; p^\circ_h|_E\in P_0(E),\; \forall E\in\mathcal{E}_h\Big\} ,
\]
whose functions have the degrees of freedom  only in the element interiors.
Similarly, the discrete vector weak function space $\mathbf{V}_h$ can be defined as
\begin{equation}
\mathbf{V}_h=\Big \{\mathbf{u}_h=\{\mathbf{u}^\circ_h, \mathbf{u}^\partial_h\}\; :\; \mathbf{u}^\circ_h|_E\in P_0(E)^2,
\;\mathbf{u}^\partial_h|_e \in P_0(e)^2,\; \forall E\in\mathcal{E}_h,\; e \in \Gamma_h\Big \} . 
\label{Vh}
\end{equation}
For any element $E$, we denote its centroid by $(x_E, y_E)$. Let $X=x-x_E$ and $Y=y-y_E$. Then, we can define the local 
Arbogast-Correa (AC) space \cite{ArbogastCorreaSIAM2016} of the lowest degree as
\[
  AC_0(E) = \text{span}
  \left\{
    \left[ \begin{array}{r} 1\\ 0 \end{array} \right],\;
    \left[ \begin{array}{r} 0\\ 1 \end{array} \right],\;
    \left[ \begin{array}{r} X\\ Y \end{array} \right],\;
    \mathcal{P}_E \left[ \begin{array}{r} \hat{x}\\ -\hat{y} \end{array} \right]
  \right\},
\]
where $\mathcal{P}_E$ is the Piola operator transforming a vector-valued function $\hat{\mathbf{v}}$ defined
on the reference element $\hat{E}$
into a vector-valued function $\mathbf{v}$ defined on $E$ as
\[
\mathbf{v}(\mathbf{x}) = \mathcal{P}_E (\hat{\mathbf{v}})(\mathbf{x})= \frac{1}{\det(F_E^{'})}
F_E^{'} \hat{\mathbf{v}}(F_E^{-1}(\mathbf{x})),
\]
where $F_E$ is the mapping from $\hat{E}$ to $E$ and $F_E^{'}$ is the Jacobian matrix of $F_E$.
On the three-dimensional domain,
we use the Arbogast-Tao (AT) {\cite{ArboTao_NumerMath_2019}} vector-valued finite element spaces.
The local $AT_0(E)$ space of the lowest order is defined as
$$AT_0(E) =   \left\{
  \displaystyle
  \left[ \begin{array}{r} 1 \\ 0 \\ 0 \end{array} \right], \;\;
  \left[ \begin{array}{r} 0 \\ 1 \\ 0 \end{array} \right], \;\;
  \left[ \begin{array}{r} 0 \\ 0 \\ 1 \end{array} \right], \;\;
  \left[ \begin{array}{r} X \\ Y \\ Z \end{array} \right], \;\;
  \mathcal{P}_E \left[ \begin{array}{r} \hat{x} \\ -\hat{y} \\ 0 \end{array} \right], \;\;
  \mathcal{P}_E \left[ \begin{array}{r} 0 \\ \hat{y} \\ -\hat{z} \end{array} \right] \right\},
$$
where $ Z=z-z_E $ and $z_E$ is the $z$-coordinate of the centroid.

Having defined the weak function spaces, we can now define the discrete weak gradient and divergence operators.
The discrete weak gradient operator $\nabla_w: W_h \rightarrow \mathcal{AC}_0(\mathcal{E}_h)$ is defined elementwise as
\begin{equation*}
  (\nabla_w u_h, \mathbf{w})_E
  = \langle u^\partial_h, \mathbf{w} \cdot \mathbf{n}\rangle_{\partial E}
  - ( u^\circ_h , \nabla \cdot \mathbf{w})_E,
  \quad \forall \mathbf{w} \in AC_0(E),\quad \forall E \in \mathcal{E}_h 
\end{equation*}
where $\mathbf{n}$ is the unit outward normal to $\partial E$ and $(\cdot, \cdot)_E$ and $\langle \cdot, \cdot \rangle_{\partial E}$ are the $L^2$ inner product on $E$ and $\partial E$, respectively.
For vector-valued functions, the discrete weak gradient takes the form
\begin{equation}
  \displaystyle
   (\nabla_w\mathbf{v}_h, W)_E
  = \langle \mathbf{v}^\partial_h, W \mathbf{n}\rangle_{\partial E}
  - ( \mathbf{v}^\circ_h ,\nabla \cdot W)_E,
  \quad
  \forall W \in AC_0(E)^2,\quad \forall E\in\mathcal{E}_h
    \label{WG_Elas_WGrad}
\end{equation}
where $AC_0(E)^2$ is a matrix-valued function space whose row vectors are in $AC_0(E)$
and $\nabla_w \mathbf{v}_h|_E \in AC_0(E)^2$.
The discrete weak divergence operator $\nabla_w \cdot: \mathbf{V}_h \to \mathcal{P}_0(\mathcal{E}_h)$ 
is defined as
\begin{equation}
   (\nabla_w \cdot \mathbf{v}_h, w )_E
  = \langle \mathbf{v}^\partial_h , w \mathbf{n}\rangle_{\partial E}
  - ( \mathbf{v}^\circ_h , \nabla w)_E,
  \quad
  \forall w \in P_0(E), \quad \forall E \in \mathcal{E}_h.
      \label{WG_Elas_WDiv}
\end{equation}
Analytical expressions for the discrete weak gradient and divergence operators
can be found in \cite{HuangWang_CiCP_2015,LiuTavWang_JCP_2018}.

The WG discretization of (\ref{poro_variationalform}) is to find
$ \mathbf{u}_h(\cdot,t)\in \mathbf{V}_h $, $ p_h(\cdot,t)\in W_h$, $0 < t \le T$, such that
$\mathbf{u}_h|_{\partial \Omega}= \mathbf{u}_D^h$,
$ p_h|_{\partial \Omega}=p_D^h $, where $\mathbf{u}_D^h$ and $p_D^h$ are
the $L^2$-projections of $\mathbf{u}_D$ and $p_D$ on $\partial \Omega$, and
\begin{equation}
\begin{cases}
\displaystyle
\mu \sum_{E\in\mathcal{E}_h} (\nabla_w\mathbf{u}_h,\nabla_w\mathbf{v}_h)_E
+(\lambda+\mu) \sum_{E\in\mathcal{E}_h} (\nabla_w\cdot\mathbf{u}_h,\nabla_w\cdot\mathbf{v}_h)_E
\\ 
\displaystyle \qquad \qquad \qquad 
- \alpha \sum_{E \in \mathcal{E}_h}  (p_h^{\circ}, \nabla_{w}\cdot\mathbf{v}_h)_E,
    = \sum_{E\in\mathcal{E}_h}(\mathbf{f},\mathbf{v}_h^{\circ})_E, \quad \forall \mathbf{v}_h \in \mathbf{V}_h^0
\\
\displaystyle
-\alpha \sum_{E \in \mathcal{E}_h}  (\nabla_{w}\cdot\frac{\partial \mathbf{u}_h}{\partial t},q_h^{\circ})_E
- c_0 \sum_{E \in \mathcal{E}_h} (\frac{\partial p_h^{\circ}}{\partial t}, q_h^{\circ})_E 
\\
\displaystyle
\qquad \qquad \qquad 
- \sum_{E\in\mathcal{E}_h}(\kappa\nabla_wp_h,\nabla_wq_h)_E
= - \sum_{E\in\mathcal{E}_h}(s,q_h^{\circ})_E,\quad \forall q_h \in {W}_h^0 .
\end{cases}
\label{EqnSemiDisc2}
\end{equation}

Assume a constant time step size $\Delta t$ is used and denote the time instants by $t_n = n \Delta t$, $ n = 0, 1, ...$.
Then, applying the implicit Euler scheme to (\ref{EqnSemiDisc2}) we obtain the fully discrete WG scheme for (\ref{EqnPoroElas}) 
(at $t = t_n, \, n\geq 1$) as
\begin{equation}
\begin{cases}
\displaystyle
\mu \sum_{E\in\mathcal{E}_h} (\nabla_w\mathbf{u}_h^n,\nabla_w\mathbf{v}_h)_E
+(\lambda+\mu) \sum_{E\in\mathcal{E}_h} (\nabla_w\cdot\mathbf{u}_h^n,\nabla_w\cdot\mathbf{v}_h)_E
\\
\displaystyle
\qquad \qquad \qquad 
- \alpha \sum_{E \in \mathcal{E}_h}  (p_h^{\circ,n}, \nabla_{w}\cdot\mathbf{v}_h)_E,
= \sum_{E\in\mathcal{E}_h}(\mathbf{f}^n,\mathbf{v}_h^{\circ})_E, \quad \forall \mathbf{v}_h \in \mathbf{V}_h^0
\\
\displaystyle
-\alpha \sum_{E \in \mathcal{E}_h}  (\nabla_{w}\cdot\mathbf{u}_h^n,q_h^{\circ})_E
- c_0 \sum_{E \in \mathcal{E}_h} (p_h^{\circ,n}, q_h^{\circ})_E 
- \Delta t \sum_{E\in\mathcal{E}_h}(\kappa\nabla_w p_h^n,\nabla_wq_h)_E
\\
\displaystyle
\qquad \qquad \qquad 
= - \Delta t \sum_{E\in\mathcal{E}_h}(s^n,q_h^{\circ})_E
    -\alpha \sum_{E \in \mathcal{E}_h}  (\nabla_{w}\cdot\mathbf{u}_h^{n-1},q_h^{\circ})_E
\\
\displaystyle
\qquad \qquad \qquad \qquad \qquad \qquad \qquad
- c_0 \sum_{E \in \mathcal{E}_h} (p_h^{\circ,n-1}, q_h^{\circ})_E
    ,\quad \forall q_h \in {W}_h^0 .
\end{cases}
\label{EqnFullDisc1}
\end{equation}
The above equation forms a system of algebraic equations for $\mathbf{u}_h^n$ and $p_h^n$.
Our main goal of this work is to develop parameter-robust and efficient preconditioners for the fast iterative solution of the system. 

We first consider the situation with homogeneous Dirichlet conditions, i.e., $\mathbf{u}_D \equiv 0$ and $p_D \equiv 0$.
In this case, $\mathbf{u}_h^n \in \mathbf{V}_h^0$ and $p_h^n \in W_h^0$. 
Notice that any function $\mathbf{v}_h \in \mathbf{V}_h^0$ can be represented by a two-block vector
$[\mathbf{v}^{\circ},\mathbf{v}^{\partial}]^T$ under a suitable basis for $\mathbf{V}_h^0$, where $\mathbf{v}^{\circ}$
and $\mathbf{v}^{\partial}$ are the vectors representing the degrees of freedom for the interiors and edges of the mesh elements,
respectively. Hereafter, we use $\mathbf{v}_h$ to denote interchangeably a function in $\mathbf{V}_h^0$ and
its vector representation $[\mathbf{v}^{\circ},\mathbf{v}^{\partial}]^T$. Similarly, any function $q_h$ in $W_h^0$ can
be represented by the vector $\mathbf{q}_h = [\mathbf{q}^{\circ},\mathbf{q}^{\partial}]^T$. With this notation, we can cast
(\ref{EqnFullDisc1}) into a matrix-vector form as
\begin{equation}
    \begin{bmatrix}
        \mu A_1  + ( \lambda + \mu ) A_0 & \alpha B^T \\
        \alpha B & -D
    \end{bmatrix}
    \begin{bmatrix}
        \mathbf{u}_h \\
        \mathbf{p}_h
    \end{bmatrix}
    =
    \begin{bmatrix}
        \mathbf{b}_1 \\
        \mathbf{b}_2
    \end{bmatrix},
    \label{2by2Scheme_matrix}
\end{equation}
where the matrices $A_0$, $A_1$, $B$, and $D$ are defined as
\begin{align}
& \mathbf{v}_h^T A_0 \mathbf{u}_h = \sum_{E\in\mathcal{E}_h} (\nabla_w\cdot\mathbf{u}_h,\nabla_w\cdot\mathbf{v}_h)_E,
\quad \forall \mathbf{u}_h, \mathbf{v}_h \in \mathbf{V}_h^0
\label{A0-1}
\\
& \mathbf{v}_h^T A_1 \mathbf{u}_h = \sum_{E\in\mathcal{E}_h} (\nabla_w\mathbf{u}_h,\nabla_w\mathbf{v}_h)_E,
\quad \forall \mathbf{u}_h, \mathbf{v}_h \in \mathbf{V}_h^0
\label{A1-1}
\\
& \mathbf{q}_h^T B \mathbf{u}_h = - \sum_{E \in \mathcal{E}_h}  (\nabla_{w}\cdot\mathbf{u}_h,q_h^{\circ})_E,
\quad \forall \mathbf{u}_h \in \mathbf{V}_h^0, \quad \forall q_h \in W_h^0
\label{B-1}
\\
& \mathbf{q}_h^T D \mathbf{q}_h = c_0 \sum_{E \in \mathcal{E}_h} (p_h^{\circ}, q_h^{\circ})_E 
+ \Delta t \sum_{E\in\mathcal{E}_h}(\kappa\nabla_w p_h,\nabla_wq_h)_E,
\quad \forall p_h, q_h \in W_h^0 .
\label{D-1}
\end{align}
The right-hand side vectors $\mathbf{b}_1$ and $\mathbf{b}_2$ are defined accordingly.

For the general situation with not-all-zero Dirichlet boundary conditions, the system (\ref{EqnFullDisc1}) can still be
cast in the form (\ref{2by2Scheme_matrix}) except that the right-hand side vectors should be modified to include
the left-hand side terms associated with the degrees of freedom on the domain boundary for both the displacement and
pressure.

In the following we explore the structures of the matrices in (\ref{2by2Scheme_matrix})
that are useful for the development of preconditioners.
First, the absence of $q_h^{\partial}$ in the right-hand side of (\ref{B-1}) implies that $B$
has the structure
\begin{equation}
B = \begin{bmatrix} B^{\circ} \\ 0 \end{bmatrix},
\label{B-2}
\end{equation}
where the row number of the zero block is equal to the number of the degrees of freedom of $q_h^{\partial}$.
The above equation implies that $B$ is singular and (\ref{B-1}) can be rewritten as
\begin{equation}
    (\mathbf{q}_h^{\circ})^T B^{\circ} \mathbf{u}_h
    = - \sum_{E \in \mathcal{E}_h}  (\nabla_{w}\cdot\mathbf{u}_h,q_h^{\circ})_E,
\quad \forall \mathbf{u}_h \in \mathbf{V}_h^0, \quad \forall q_h \in W_h^0 .
\label{B-3}
\end{equation}

Secondly, define $w_h = \nabla_w \cdot \mathbf{u}_h$.
Recalling that $w_h \in \mathcal{P}_0(\mathcal{E}_h)$, we have
\[
\sum_{E \in \mathcal{E}_h} (w_h^{\circ}, q_h^{\circ})_E =
\sum_{E \in \mathcal{E}_h}  (\nabla_{w}\cdot\mathbf{u}_h,q_h^{\circ})_E,\quad
\forall q_h \in \mathcal{P}_0(\mathcal{E}_h) .
\]
This, combined with (\ref{B-3}), yields
\[
(\mathbf{q}_h^{\circ})^T M_p^{\circ} \mathbf{w}_h^{\circ} = - (\mathbf{q}_h^{\circ})^T B^{\circ} \mathbf{u}_h,
\quad \forall q_h \in \mathcal{P}_0(\mathcal{E}_h)
\]
where $M_p^{\circ}$ is the mass matrix of interior pressure. Thus, we get
\begin{equation}
\mathbf{w}_h^{\circ} = - (M_p^{\circ})^{-1} B^{\circ} \mathbf{u}_h
\quad \text{for} \quad  w_h = \nabla_w \cdot \mathbf{u}_h.
\label{w-1}
\end{equation}
Using this result and (\ref{B-3}), we can rewrite (\ref{A0-1}) into
\begin{align*}
\mathbf{u}_h^T A_0 \mathbf{v}_h & = \sum_{E\in\mathcal{E}_h} (\nabla_w\cdot\mathbf{v}_h,\nabla_w\cdot\mathbf{u}_h)_E 
= - (\mathbf{w}_h^{\circ})^T B^{\circ} \mathbf{v}_h
\\
&
= \mathbf{u}_h^T (B^{\circ})^T (M_p^{\circ})^{-1} B^{\circ} \mathbf{v}_h,
\quad \forall \mathbf{u}_h, \mathbf{v}_h \in \mathbf{V}_h^0
\end{align*}
which gives
\begin{equation}
    A_0 = (B^{\circ})^T (M_p^{\circ})^{-1} B^{\circ} .
\label{A0-2}
\end{equation}
Since the row number of $B^{\circ}$ (which is equal to the number of the degrees of freedom of $p_h^{\circ}$)
is less than its column number (which is equal to the number of the degrees of freedom of $\mathbf{u}_h$),
$B^{\circ}$ does not have full rank and (\ref{A0-2}) implies that $A_0$ is singular.

Thirdly, the matrix $D$ has the expression as
\begin{equation}
D = c_0 \begin{bmatrix}
    M_p^{\circ} & 0 \\[0.1in]
    0 & 0
\end{bmatrix}
+ { \kappa \Delta t } A_p,
\label{D-2}
\end{equation}
where $A_p$ is the stiffness matrix of the Laplacian operator for pressure. Thus, $D$ is becoming singular as $\Delta t \to 0$.
This is a distinct feature of the WG discretization that requires special treatments in the development of preconditioners.
Similarly, the right-hand side $\mathbf{b}_2$ has the structure
\begin{equation}
\mathbf{b}_2 = \begin{bmatrix} \mathbf{b}_2^{\circ} \\ 0 \end{bmatrix} .
\label{b2-1}
\end{equation}

Finally, it is worth emphasizing that $A_1$, $A_p$, and $D$ are symmetric and positive definite (SPD), $A_0$ is
symmetric and positive semi-definite, and (\ref{2by2Scheme_matrix}) is a saddle-point problem whose coefficient matrix
has both positive and negative eigenvalues and is indefinite.

We now turn our attention to (\ref{2by2Scheme_matrix}) and would like to develop parameter-robust and efficient preconditioners for
its fast iterative solution. In our computation, we use the GMRES method
since it works for indefinite systems and with non-symmetric preconditioners.
Other Krylov subspace methods with similar properties can also be used.
Recall that (\ref{2by2Scheme_matrix}) is a saddle-point system. In the past, iterative solution and
preconditioning for saddle-point systems have been studied extensively; e.g.,
see \cite{BenziGolubLiesen-2005,Benzi2008} and references therein.
We use the Schur complement preconditioning in this work. 
While this technique has been studied greatly so far,
its application to (\ref{2by2Scheme_matrix}) still poses challenges due to the distinct features of the system.
First, the (2,2) block of (\ref{2by2Scheme_matrix}), $D$, is nonzero. As we can see in (\ref{D-2}), it is not exactly the mass
matrix nor the stiffness matrix of the Laplacian operator, and is becoming singular as $\Delta t \to 0$.
Moreover, the (1,2) and (2,1) blocks are singular.
Furthermore, the (1,1) block (the leading block)
is becoming singular as $\lambda \to \infty$ while $\mu$ stays bounded (the locking case).
This last feature can be seen more clearly if we divide both sides of (\ref{2by2Scheme_matrix}) by $( \lambda + \mu )$, 
\begin{equation}
    \mathcal{A}_2 
    \begin{bmatrix}
        \mathbf{u}_h \\
        \mathbf{p}_h
    \end{bmatrix}
    = \frac{\epsilon}{\mu}
    \begin{bmatrix}
        \mathbf{b}_1 \\
        \mathbf{b}_2
    \end{bmatrix},
    \qquad \mathcal{A}_2 = \begin{bmatrix}
        \epsilon A_1  +  A_0 & \frac{\alpha \epsilon}{\mu } B^T \\[0.05in]
        \frac{\alpha \epsilon}{\mu} B & -\frac{\epsilon}{\mu }D
    \end{bmatrix}, 
    \label{2by2Scheme_matrix2}
\end{equation}
where $\epsilon = \frac{\mu}{ \lambda + \mu }$.
Recall that $A_0$ is positive semi-definite. As $\lambda \to \infty$ (i.e., $\epsilon \to 0$), the (1,1) block becomes singular.
These features make (\ref{2by2Scheme_matrix2}) distinct from other saddle-point systems and existing preconditioners
for general saddle-point systems do not work for (\ref{2by2Scheme_matrix2}). In the next two sections, we study
the Schur complement preconditioning for (\ref{2by2Scheme_matrix2}).
Although our study focuses on
locking situation (i.e., $\epsilon \to 0$), the analysis applies to non-locking situations as well.

\section {Two-field Schur complement preconditioning}
\label{SEC:2by2}

In this section we study the two-field Schur complement preconditioning for (\ref{2by2Scheme_matrix2}).

\subsection{Two-field Schur complement preconditioners}
\label{SEC:2-field-poroe}

We start with noticing that part of the second block of (\ref{2by2Scheme_matrix2}) becomes zero as $\Delta t \to 0$; see (\ref{B-2}) and (\ref{D-2}). This is a distinct feature of the WG discretization.
In principle, we can work directly with (\ref{2by2Scheme_matrix2}) but this leads to an undesired upper bound
proportional to $1/\Delta t$ even for the case with $c_0 > 0$ (cf. (\ref{pro:2-field-1})).
To circumvent this, we first eliminate the degrees of freedom of pressure on edges, develop preconditioners for the resulting reduced system, and then use them to derive preconditioners for the original system (\ref{2by2Scheme_matrix2}).
To this end, we partition the pressure vector and the Laplacian operator for pressure as 
\begin{equation}
\mathbf{p}_h = \begin{bmatrix} \mathbf{p}_h^{\circ} \\[0.05in]
\mathbf{p}_h^{\partial}\end{bmatrix},\quad
A_p = \begin{bmatrix} A_p^{\circ\circ} & A_p^{\circ\partial}
\\[0.05in] A_p^{\partial\circ} & A_p^{\partial\partial} \end{bmatrix} .
\label{Ap-1}
\end{equation}
Since $A_p$ is SPD, so are $A_p^{\circ\circ}$, $A_p^{\partial\partial}$, and 
$A_p^{\circ\circ}-A_p^{\circ\partial} (A_p^{\partial\partial})^{-1} A_p^{\partial\circ}$.
From (\ref{B-2}), (\ref{D-2}), and (\ref{b2-1}) we can rewrite the second block of (\ref{2by2Scheme_matrix2}) as
\[
\begin{cases}
\displaystyle \frac{\alpha \epsilon}{\mu} B^{\circ}  \mathbf{u}_h
- \frac{\epsilon }{\mu} \left [ (c_0 M_p^{\circ} + { \kappa} \Delta t A_p^{\circ\circ} ) \mathbf{p}_h^{\circ} + {\kappa} \Delta t A_p^{\circ\partial}  \mathbf{p}_h^{\partial} \right] = \frac{\epsilon}{\mu} \mathbf{b}_2^{\circ},
\\[10pt]
\displaystyle - \frac{\epsilon}{\mu} \left [
{\kappa} \Delta t A_p^{\partial\circ} \mathbf{p}_h^{\circ}
+ { \kappa} \Delta t A_p^{\partial\partial} \mathbf{p}_h^{\partial}
\right ] = 0 .
\end{cases}
\]
Solving the second equation for $\mathbf{p}_h^{\partial}$ and inserting it into the first equation, we obtain
\begin{equation}
\begin{cases}
\displaystyle \frac{\alpha \epsilon}{\mu} B^{\circ} \mathbf{u}_h
- \frac{\epsilon}{\mu} \left [ c_0 M_p^{\circ} + { \kappa}  \Delta t \left (A_p^{\circ\circ}-A_p^{\circ\partial} (A_p^{\partial\partial})^{-1}A_p^{\partial\circ} \right ) \right ] \mathbf{p}_h^{\circ} = \frac{\epsilon}{\mu} \mathbf{b}_2^{\circ} ,
\\[10pt]
\displaystyle \mathbf{p}_h^{\partial} = - (A_p^{\partial\partial})^{-1} A_p^{\partial\circ} \mathbf{p}_h^{\circ} .
\end{cases}
\label{pressure-2}
\end{equation}
Using this, we can rewrite (\ref{2by2Scheme_matrix2}) into
\begin{equation}
    \tilde{\mathcal{A}}_{2}
    \begin{bmatrix}
        \mathbf{u}_h \\
        \mathbf{p}_h^{\circ}
    \end{bmatrix}
    = \frac{\epsilon}{\mu}
    \begin{bmatrix}
        \mathbf{b}_1 \\
        \mathbf{b}_2^{\circ}
    \end{bmatrix},
    \label{2by2Scheme_matrix2-2}
\end{equation}
where
\begin{equation}
    \label{D-3}
\tilde{\mathcal{A}}_{2} = 
    \begin{bmatrix}
        \epsilon A_1  +  A_0 & \frac{\alpha \epsilon}{\mu } (B^{\circ})^T \\[0.05in]
        \frac{\alpha \epsilon}{\mu} B^{\circ} & -\frac{\epsilon}{\mu }\tilde{D}
    \end{bmatrix},
    \quad
    \tilde{D} = c_0 M_p^{\circ} + { \kappa} \Delta t \left (A_p^{\circ\circ}-A_p^{\circ\partial} (A_p^{\partial\partial})^{-1}A_p^{\partial\circ} \right ) .
\end{equation}
For this system, the Schur complement is
\begin{equation}
\label{S2-1}
\tilde{S}_2 = \frac{\epsilon}{\mu}\tilde{D} + \left (\frac{\alpha \epsilon}{\mu }\right )^2 B^{\circ}\left (\epsilon A_1  +  A_0\right )^{-1}(B^{\circ})^T .
\end{equation}
The ideal block triangular preconditioner is
\begin{align}
\tilde{\mathcal{P}}_{2,ideal} = 
    \begin{bmatrix}
         \epsilon A_1  +  A_0 & \frac{\alpha \epsilon }{\mu } (B^{\circ})^T \\[0.04in]
        0 & - \tilde{S}_2
    \end{bmatrix}.
    \label{P2-1}
\end{align}
It is known (e.g., see \cite{Benzi2008,MurphyGolubWathen_SISC_2000} and Appendix~\ref{appendix_A}) that
$\tilde{\mathcal{P}}_{2,ideal}^{-1} \tilde{\mathcal{A}}_{2}$ has a single eigenvalue 1.
This means that it only takes a few steps if an iterative solver is applied to the preconditioned system.
Unfortunately, $\tilde{\mathcal{P}}_{2,ideal}$ is not practical since it requires to compute
the Schur complement $\tilde{S}_{2}$ (a dense matrix) and its inverse.
In practice, $\tilde{S}_2$ is replaced by an approximation; see the discussion
on Schur complement preconditioning for general saddle-point systems in Appendix~\ref{appendix_A}.
To this end, we establish some bounds for $\tilde{S}_2$ in the following.

\begin{lem}
\label{lem:3.1}
Let $A$ be an arbitrary SPD matrix of size $m$ and $C$ be an arbitrary matrix with a proper number of columns.
Then,
\[
    \| C A^{-1} C^T \| = \sup_{\mathbf{x}\neq 0}
    \frac{\mathbf{x}^T C^T C \mathbf{x}}{\mathbf{x}^T A  \mathbf{x}} .
\]
\end{lem}

\begin{proof}
Notice that the matrices
\[
C A^{-1} C^T = (C A^{-\frac{1}{2}}) (C A^{-\frac{1}{2}})^T, \quad
(C A^{-\frac{1}{2}})^T (C A^{-\frac{1}{2}}) = A^{-\frac{1}{2}} C^T C A^{-\frac{1}{2}}
\]
have the same non-zero
eigenvalues. Then, we have
\begin{align*}
\| C A^{-1} C^T \|
& = \sup_{\mathbf{x}\neq 0} \frac{\mathbf{x}^T A^{-\frac{1}{2}} C^T C A^{-\frac{1}{2}} \mathbf{x}}{\mathbf{x}^T \mathbf{x}}
= \sup_{\mathbf{x}\neq 0} \frac{\mathbf{x}^T C^T C \mathbf{x}}{\mathbf{x}^T A \mathbf{x}} .
\end{align*}
\end{proof}

\begin{pro}
\label{pro:2-field}
\begin{align}
\frac{\epsilon}{\mu}\tilde{D} \le \tilde{S}_2 \le
\begin{cases} \displaystyle
\frac{\epsilon}{\mu}\tilde{D} \left (1 + \frac{\alpha^2 \epsilon}{\mu c_0} \right ), \quad & \text{ for } c_0 > 0
\\[0.1in]
\displaystyle \frac{\epsilon}{\mu}\tilde{D} \left (1 + \frac{\alpha^2 \epsilon \lambda_{\max}(M_p^{\circ})}{\mu {\kappa} \Delta t \lambda_{\min}\left (A_p^{\circ\circ}-A_p^{\circ\partial} (A_p^{\partial\partial})^{-1}A_p^{\partial\circ} \right ) } \right ), \quad & \text{ for } c_0 = 0 
\end{cases}
\label{pro:2-field-1}
\end{align}
where $\le$ is in the negative semi-definite sense and $\lambda_{\max}(M_p^{\circ})$ denotes the maximum eigenvalue of
$M_p^{\circ}$, and $\lambda_{\min}\left (A_p^{\circ\circ}-A_p^{\circ\partial} (A_p^{\partial\partial})^{-1}A_p^{\partial\circ} \right )$
denotes the minimum eigenvalue of $A_p^{\circ\circ}-A_p^{\circ\partial} (A_p^{\partial\partial})^{-1}A_p^{\partial\circ}$.
\end{pro}

\begin{proof}
The left inequality of (\ref{pro:2-field-1}) follows from (\ref{S2-1}) and the fact that $\epsilon A_1 + A_0$ is SPD.

For the right inequality, we first consider the case with $c_0 > 0$. From (\ref{A0-2}), (\ref{D-3}), and Lemma~\ref{lem:3.1}  we have
\begin{align*}
& \sup_{\mathbf{v} \neq 0} \frac{\mathbf{v}^T B^{\circ}  (\epsilon A_1 + A_0)^{-1} (B^{\circ})^T \mathbf{v}}{\mathbf{v}^T \tilde{D} \mathbf{v}}
\le \sup_{\mathbf{v} \neq 0} \frac{\mathbf{v}^T B^{\circ}  (\epsilon A_1 + A_0)^{-1} (B^{\circ} )^T \mathbf{v}} {c_0 \mathbf{v}^T M_{p}^{\circ} \mathbf{v}}
\\
& = \frac{1}{c_0} \| (M_{p}^{\circ})^{-\frac12} B^{\circ}  (\epsilon A_1 + A_0)^{-1} (B^{\circ})^T (M_{p}^{\circ})^{-\frac12} \|
= \frac{1}{c_0} \sup_{\mathbf{u} \neq 0} \frac{\mathbf{u}^T (B^{\circ})^T (M_{p}^{\circ})^{-1} B^{\circ} \mathbf{u}}{\mathbf{u}^T (\epsilon A_1 + A_0) \mathbf{u}} 
\\
& = \frac{1}{c_0} \sup_{\mathbf{u} \neq 0} \frac{\mathbf{u}^T A_0 \mathbf{u}}{\mathbf{u}^T (\epsilon A_1 + A_0) \mathbf{u}}
\le \frac{1}{c_0} .
\end{align*}
Combining this with (\ref{S2-1}), we have
\[
\tilde{S}_2 \le \frac{\epsilon}{\mu}\tilde{D} \left (1 + \frac{\alpha^2 \epsilon}{\mu c_0} \right ) .
\]

When $c_0 = 0$, from (\ref{D-3}) and Lemma~\ref{lem:3.1} we have
\begin{align*}
& \sup_{\mathbf{v} \neq 0} \frac{\mathbf{v}^T B^{\circ}  (\epsilon A_1 + A_0)^{-1} (B^{\circ})^T \mathbf{v}}{\mathbf{v}^T \tilde{D} \mathbf{v}}
 = \sup_{\mathbf{v} \neq 0} \frac{\mathbf{v}^T B^{\circ}  (\epsilon A_1 + A_0)^{-1} (B^{\circ} )^T \mathbf{v}} {{ \kappa} \Delta t \mathbf{v}^T \left (A_p^{\circ\circ}-A_p^{\circ\partial} (A_p^{\partial\partial})^{-1}A_p^{\partial\circ} \right ) \mathbf{v}}
\\
& \le \frac{1}{{ \kappa} \Delta t \lambda_{min}\left (A_p^{\circ\circ}-A_p^{\circ\partial} (A_p^{\partial\partial})^{-1}A_p^{\partial\circ} \right )}
\| B^{\circ}  (\epsilon A_1 + A_0)^{-1} (B^{\circ} )^T \| 
\\
& = \frac{1}{{ \kappa} \Delta t \lambda_{min}\left (A_p^{\circ\circ}-A_p^{\circ\partial} (A_p^{\partial\partial})^{-1}A_p^{\partial\circ} \right )}
\sup_{\mathbf{u} \neq 0} \frac{\mathbf{u}^T (B^{\circ})^T B^{\circ} \mathbf{u}}{\mathbf{u}^T (\epsilon A_1 + A_0) \mathbf{u}} 
\\
& \le \frac{\lambda_{\max}(M_p^{\circ})}{{ \kappa} \Delta t \lambda_{\min}\left (A_p^{\circ\circ}-A_p^{\circ\partial} (A_p^{\partial\partial})^{-1}A_p^{\partial\circ} \right )}
\sup_{\mathbf{u} \neq 0} \frac{\mathbf{u}^T (B^{\circ})^T (M_p^{\circ})^{-1} B^{\circ} \mathbf{u}}{\mathbf{u}^T (\epsilon A_1 + A_0) \mathbf{u}} 
\\
& \le \frac{\lambda_{\max}(M_p^{\circ})}{{ \kappa} \Delta t \lambda_{\min}\left (A_p^{\circ\circ}-A_p^{\circ\partial} (A_p^{\partial\partial})^{-1}A_p^{\partial\circ} \right )} .
\end{align*}
Thus, we have
\[
\tilde{S}_2 \le \frac{\epsilon}{\mu}\tilde{D} \left (1 + \frac{\alpha^2 \epsilon \lambda_{\max}(M_p^{\circ})}{\mu { \kappa} \Delta t \lambda_{\min}\left (A_p^{\circ\circ}-A_p^{\circ\partial} (A_p^{\partial\partial})^{-1}A_p^{\partial\circ} \right ) } \right ) .
\]
\end{proof}

For quasi-uniform meshes, it is known (see, e.g., \cite{LinLiuFarrah_JCAM_2015})
that
\[
\lambda_{\max}(M_p^{\circ}) = \mathcal{O}(h^2), \quad 
\lambda_{\min}\left (A_p^{\circ\circ}-A_p^{\circ\partial} (A_p^{\partial\partial})^{-1}A_p^{\partial\circ} \right )
\approx \lambda_{\min}(A_p) = \mathcal{O}(h^2).
\]
In this case, the factor on the right-hand side of (\ref{pro:2-field-1}) for $c_0 = 0$ becomes
\[
1 + \frac{\alpha^2 \epsilon \lambda_{\max}(M_p^{\circ})}{{ \kappa} \Delta t \mu \lambda_{\min}\left (A_p^{\circ\circ}-A_p^{\circ\partial} (A_p^{\partial\partial})^{-1}A_p^{\partial\circ} \right )}
\le 1 + \frac{ C \alpha^2 \epsilon }{{\kappa} \Delta t \mu},
\]
where $C$ is a constant.

Proposition~\ref{pro:2-field} shows that $\tilde{S}_2$ is spectrally equivalent to $\frac{\epsilon}{\mu}\tilde{D}$.
Thus, we define a block triangular preconditioner for (\ref{2by2Scheme_matrix2-2}) as
\begin{equation}
\label{P2-2}
\tilde{\mathcal{P}}_2 = \begin{bmatrix}
         \epsilon A_1  +  A_0 & \frac{\alpha \epsilon}{\mu} (B^{\circ})^T \\[0.04in]
        0 & -\frac{\epsilon}{\mu}\tilde{D}
    \end{bmatrix} .
\end{equation}
From Lemma~\ref{lem:SPP} and Proposition~\ref{pro:2-field}, the eigenvalues of the preconditioned matrix 
$\tilde{\mathcal{P}}_2^{-1} \tilde{\mathcal{A}}_2$ are bounded by
\begin{equation}
\label{2-field-eigen}
1 \le \lambda_j \le
\begin{cases} \displaystyle 1 + \frac{\alpha^2 \epsilon}{\mu c_0}, \quad & \text{ for } c_0 > 0
\\[0.05in]
\displaystyle 1 + \frac{\alpha^2 \epsilon \lambda_{\max}(M_p^{\circ})}{\mu { \kappa} \Delta t \lambda_{\min}\left (A_p^{\circ\circ}-A_p^{\circ\partial} (A_p^{\partial\partial})^{-1}A_p^{\partial\circ} \right ) }, \quad & \text{ for } c_0 = 0 . 
\end{cases}
\end{equation}
For $c_0 > 0$, the upper bound is independent of
{ $\kappa \Delta t$ and $h$} and tends to 1
as $\epsilon \to 0$. Thus, we can expect $\tilde{\mathcal{P}}_2$ to work well and be robust about
the locking parameter $\lambda$, time step $\Delta t$, { permeability $\kappa$,} and mesh size $h$. For the case $c_0 = 0$,
the upper bound is more complicated. Nevertheless, it tends to 1 as $\epsilon \to 0$
and remains at a reasonable size for all parameter variations except when ${ \kappa}\Delta t \ll \epsilon$.
Therefore, we can expect the preconditioner to work well also for $c_0 = 0$
except when ${ \kappa} \Delta t \ll \epsilon$.
{In the above analysis, we have seen that $\kappa$ and $\Delta t$ act together as a single
effective parameter, through their product $\kappa \Delta t$; changes to one can be offset by inverse changes
to the other. For this reason, we fix $\kappa = 1$ and only consider changes to $\Delta t$
in our numerical examples.}

Note that (\ref{2-field-eigen}) also works for $c_0 = C \epsilon$ for some positive constant $C$, a case considered by
several researchers (e.g., see \cite{Lee-SISC-2017}). In this case, the eigenvalues are bounded
by $1+ \alpha^2/C \mu$.

The preconditioner $\tilde{\mathcal{P}}_2$ is for the linear system (\ref{2by2Scheme_matrix2-2}). It is more convenient
and economic to work directly with the original system (\ref{2by2Scheme_matrix2}) that does not involve
$(A_p^{\partial \partial})^{-1}$.
Recall the reduced system (\ref{2by2Scheme_matrix2-2}) involves unknown variables $\mathbf{u}_h$ and $\mathbf{p}_h^{\circ}$.
Using the relation between $\mathbf{p}_h^{\circ}$ and $\mathbf{p}_h^{\partial}$ (cf. the second equation of (\ref{pressure-2})),
we can rewrite (\ref{2by2Scheme_matrix2-2}) back into the system (\ref{2by2Scheme_matrix2}) for unknown variables
$\mathbf{u}_h$ and $\mathbf{p}_h = [\mathbf{p}_h^{\circ},\mathbf{p}_h^{\partial}]$.
Similarly, from $\tilde{\mathcal{P}}_2$ we can obtain
\begin{equation}
\label{P2}
\mathcal{P}_2 = \begin{bmatrix}
         \epsilon A_1  +  A_0 & \frac{\alpha \epsilon}{\mu} B^T \\[0.04in]
        0 & -\frac{\epsilon}{\mu} D
    \end{bmatrix} .
\end{equation}
It is not difficult to show that the eigenvalues of $\mathcal{P}_2^{-1} \mathcal{A}_2$ are given by 1 and those of $\tilde{\mathcal{P}}_2^{-1} \tilde{\mathcal{A}}_2$ and thus, bounded by (\ref{2-field-eigen}) as well.

For efficiency, we can replace $D$ by its incomplete Cholesky decomposition and obtain
\begin{equation}
\label{P2DLU}
\mathcal{P}_{2,DLU} = \begin{bmatrix}
         \epsilon A_1  +  A_0 & \frac{\alpha \epsilon}{\mu} B^T \\[0.04in]
        0 & -\frac{\epsilon}{\mu} L_D L_D^T
    \end{bmatrix} .
\end{equation}
As we can see from numerical experiments in Section~\ref{SEC:2-field-numerics}, this does not affect
the performance of the preconditioner significantly.

On the other hand, our limited experience shows that replacing the leading block $\epsilon A_1  +  A_0$ with
its incomplete Cholesky decomposition can worsen the performance of the preconditioner significantly.
Moreover, regularization is needed to perform its incomplete Cholesky decomposition since it is becoming
singular as $\epsilon \to 0$. For this reason, we keep the leading block in its form
in $\mathcal{P}_2$ and $\mathcal{P}_{2,DLU}$. Of course, this means we need to solve
linear systems associated with $\epsilon A_1  +  A_0$ in the implementation of these preconditioners,
which can be done by direct or iterative solvers, in theory. The former is not desired when the size of the system is large.
For the latter, commonly used iterative solvers, such as
multigrid, algebraic multigrid, and preconditioned Krylov subspace methods, converge slowly for small $\epsilon$.
In the next subsection we consider the efficient iterative solution of systems associated with $\epsilon A_1  +  A_0$
with GMRES and the Schur complement preconditioning.

\subsection{Schur complement preconditioning for linear elasticity problems}
\label{SEC:linear-elasticity}

In this subsection we consider preconditioning for solving linear systems associated with $\epsilon A_1  +  A_0$.
Such a system is actually a linear elasticity problem.
To be specific, we consider
\begin{equation}
\label{elasticity-1}
(\epsilon A_1 + A_0) \mathbf{u}_h = \mathbf{b} .
\end{equation}
Recall that $A_0 = (B^{\circ})^T (M_p^{\circ})^{-1} B^{\circ}$ (cf. (\ref{A0-1})). 
Define a new variable
\[
\mathbf{w}_h = - (M_p^{\circ})^{-1} B^{\circ} \mathbf{u}_h
\]
and rewrite (\ref{elasticity-1}) into a saddle-point problem,
\begin{equation}
\label{elasticity-2}
\mathcal{A}_{2,e}
\begin{bmatrix}
    \epsilon \mathbf{u}_h \\ \mathbf{w}_h 
\end{bmatrix}
= \begin{bmatrix}
    \mathbf{b} \\ 0 
\end{bmatrix} ,
\qquad
\mathcal{A}_{2,e} = \begin{bmatrix}
    A_1 & - (B^{\circ})^T \\ - B^{\circ} & - \epsilon M_p^{\circ}
\end{bmatrix} .
\end{equation}
The Schur complement for this system is
\begin{equation}
\label{S2-e}
S_{2,e} = \epsilon M_p^{\circ} + B^{\circ} A_1^{-1} (B^{\circ})^T .
\end{equation}

\begin{pro}
\label{pro:2-field-e}
\begin{equation}
\label{S2-e-2}
\epsilon M_p^{\circ} \le    S_{2,e} \le M_p^{\circ} (\epsilon + {d}) .
\end{equation}
\end{pro}

\begin{proof}
The proof is similar to that of Proposition~\ref{pro:2-field} except that here we need to use the inequality
\begin{equation}
\mathbf{u}_h^T A_0 \mathbf{u}_h = \sum_{E}(\nabla_w \cdot \mathbf{u_h},\nabla_w \cdot \mathbf{u_h})_E \le
{d} \sum_{E}(\nabla_w \mathbf{u}_h,\nabla_w \mathbf{u}_h )_E = {d} \mathbf{u}_h^T A_1 \mathbf{u}_h .
\label{A0A1}
\end{equation}
\end{proof}

Notice that (\ref{S2-e-2}) does not imply the spectral equivalence between  $S_{2,e}$ and $M_p^{\circ}$.
Indeed, they are not spectrally equivalent since $S_{2,e}$ becomes singular as $\epsilon \to 0$ while
$M_p^{\circ}$ is independent of $\epsilon$. Based on the discussion in Appendix~\ref{appendix_A}, 
we take $\hat{S}_{2,e} = M_p^{\circ}$ (an upper bound of $S_{2,e}$) and define
\begin{equation}
\label{P2e}
\mathcal{P}_{2,e} = \begin{bmatrix} A_1 & - (B^{\circ})^T \\ 0 & - M_p^{\circ} 
\end{bmatrix} .
\end{equation}
{It is obvious that $\mathcal{P}_{2,e}$ is nonsingular. Moreover,}
from Lemma~\ref{lem:SPP}, the eigenvalues of $\mathcal{P}_{2,e}^{-1} \mathcal{A}_{2,e}$ are given by
the eigenvalues of $\hat{S}_{2,e}^{-1} S_{2,e}$, i.e.,
\begin{equation}
\label{PA2e-eigen}
    \lambda_j = \epsilon + \frac{\mathbf{w}_j^T B^{\circ} A_1^{-1} (B^{\circ})^T\mathbf{w}_j}
    {\mathbf{w}_j^T M_p^{\circ} \mathbf{w}_j}, \quad j = 1, ..., m
\end{equation}
where $m$ denotes the size of $S_{2,e}$ and $\mathbf{w}_j$'s are eigenvectors of $(B^{\circ} A_1^{-1} (B^{\circ})^T,
M_p^{\circ})$. Arranging these eigenvalues in non-decreasing order, we have
\begin{equation}
    0 < \lambda_1 = \epsilon < \lambda_2 \le \cdots \le \lambda_{m} \le {d}+\epsilon,
\end{equation}
where $\lambda_{j}$, $j = 2, ..., m$ stay away from zero as $\epsilon \to 0$ and we have used
the fact that the dimension of $\text{Null} (B^{\circ} A_1^{-1} (B^{\circ})^T)$ is one.
Thus, the eigenvalues of $\mathcal{P}_{2,e}^{-1} \mathcal{A}_{2,e}$ cluster around 1 (formed by
$\lambda_{j}$, $j = 2, ..., m$) plus an outlier $\lambda_1 = \epsilon$.
This indicates that $\mathcal{P}_{2,e}$ and $ \mathcal{A}_{2,e}$ are not spectrally equivalent.
Furthermore,
from
\begin{align*}
& (M_p^{\circ})^{-\frac{1}{2}} B^{\circ} A_1^{-1} (B^{\circ})^T (M_p^{\circ})^{-\frac{1}{2}}
= \left ( (M_p^{\circ})^{-\frac{1}{2}} B^{\circ} A_1^{-\frac{1}{2}}\right )
\left ( (M_p^{\circ})^{-\frac{1}{2}} B^{\circ} A_1^{-\frac{1}{2}} \right )^T,
\end{align*}
we see that non-zero eigenvalues of $(M_p^{\circ})^{-\frac{1}{2}} B^{\circ} A_1^{-1} (B^{\circ})^T (M_p^{\circ})^{-\frac{1}{2}}$ are the squares of the non-zero singular values of $(M_p^{\circ})^{-\frac{1}{2}} B^{\circ} A_1^{-\frac{1}{2}}$.
It is not difficult to show that the minimum non-zero singular value of $(M_p^{\circ})^{-\frac{1}{2}} B^{\circ} A_1^{-\frac{1}{2}}$ is given by
\begin{align}
\label{beta-1}
    \displaystyle
  \beta \equiv \inf_{\mathbf{p}\in \text{Null}  \left ((B^{\circ})^T(M_p^{\circ})^{-\frac{1}{2}}\right )^{\perp}}\sup_{\mathbf{u}\neq 0}
    \frac{\mathbf{p}^T (M_p^{\circ})^{-\frac{1}{2}}B^{\circ}A_1^{-\frac{1}{2}}\mathbf{u}}
      {\sqrt{\mathbf{p}^T \mathbf{p}} \sqrt{\mathbf{u}^T \mathbf{u}}},
\end{align}
which is a matrix-vector form of the well-known inf-sup condition for linear elasticity. It has been proven by Wang et al.  \cite[Lemma 1]{Wang2TavLiu_JCAM_2024} that the WG discretization satisfies the inf-sup condition
with $\beta$ being a constant independent of $h$.
Thus, the non-small eigenvalues of $\mathcal{P}_{2,e}^{-1} \mathcal{A}_{2,e}$ are contained in the interval $\left [\epsilon + \beta^2, \; \epsilon + {d} \right ]$. The radius of this interval is
{$(d-\beta^2)/2$}
which is independent of $\epsilon$ and $h$.

Convergence of Krylov subspace methods for linear systems with eigenvalues in close clusters has been studied
extensively. One analysis particularly relevant to our current situation is by Campbell et al. \cite{Campbell-1996}
which shows that if the eigenvalues of the coefficient matrix consist of a single cluster plus outliers,
then the convergence factor of GMRES is bounded by the cluster radius, while the asymptotic error constant reflects
the non-normality of the coefficient matrix and the distance of the outliers from the cluster. Applying this to our situation, we know
that the convergence factor of GMRES applied to the preconditioned matrix $\mathcal{P}_{2,e}^{-1} \mathcal{A}_{2,e}$
is bounded by { $(d-\beta^2)/2$}, which is
independent of  the mesh size $h$ and the locking parameter $\epsilon$ (or $\lambda$).

We turn our attention back to the computation of preconditioner $\mathcal{P}_{2,e}$ in (\ref{P2e}).
It requires the solution of linear systems
associated with the leading block $A_1$. Unlike (\ref{P2}) and (\ref{P2DLU}) for poroelasticity problems,
the current leading block $A_1$ is essentially the WG approximation of the Laplacian operator
and can be solved efficiently using multigrid, algebraic multigrid, and preconditioned Krylov subspace methods.
In our computation, we use the conjugate gradient (CG) method preconditioned with incomplete Cholesky decomposition
(see the results in Table~\ref{P2e-PCG-A1}).
The performance of $\mathcal{P}_{2,e}$ will be showcased numerically in the next subsection.

\subsection{Numerical experiments for two-field preconditioning}
\label{SEC:2-field-numerics}
We consider two-dimensional linear elasticity and poroelasticity examples with $\Omega = (0,1)\times (0,1)$ and a square uniform mesh with
the element side length $h$. For iterative solution of linear systems, unless stated otherwise, 
we use MATLAB's function {\em gmres} with $tol = 10^{-6}$, $restart = 30$, and the zero vector as the initial guess.
Moreover, all incomplete Cholesky decompositions are performed
using MATLAB's function {\em ichol} with threshold dropping and drop tolerance $10^{-3}$.

\subsubsection{Linear elasticity}
This linear elasticity example, adopted from \cite{Yi_JCAM2019},
{ has the exact solution and right-hand side function as}
\begin{align*}
{ 
    \mathbf{u} = 
    \begin{bmatrix}
        \sin (x) \sin(y) + \frac{x}{\lambda} \\
        \cos (x) \cos(y) + \frac{y}{\lambda}
    \end{bmatrix}, \quad }
\mathbf{f}  = 
  \begin{bmatrix}
        2 \mu \sin(x) \sin(y)
      \\ 
        2 \mu \cos(x) \cos(y) 
\end{bmatrix} .
\end{align*}
The other parameters are set as $E = 1$,
$\lambda$ taking several different values,
and $\mu$ being calculated using $E$ and $\lambda$.
It can be seen that $\nabla \cdot \mathbf{u} = \frac{2}{\lambda}$. This example
is used to simulate a nearly incompressible system.

Table~\ref{P2e-itr} shows the GMRES iterations of the preconditioned system $\mathcal{P}_{2,e}^{-1} \mathcal{A}_{2,e}$
for both locking and non-locking cases. Recall that the eigenvalues of $\mathcal{P}_{2,e}^{-1} \mathcal{A}_{2,e}$
consist of a cluster contained in the interval $\left [\epsilon + \beta^2, \; \epsilon + { d} \right ]$ plus an outlier
$\lambda = \epsilon$ (cf. (\ref{PA2e-eigen})). According to the analysis of \cite{Campbell-1996}, the convergence factor
of GMRES for the preconditioned system is bounded by the radius of the interval,
{$(d-\beta^2)/2$}, which is independent
of $\epsilon$ (or $\lambda$) and mesh size $h$. The number of GMRES iterations required to reach convergence listed in
Table~\ref{P2e-itr} stays consistently bounded by a relatively small number, which is a reflection of the behavior of GMRES
predicted by the analysis and a confirmation of the effectiveness and robustness of the preconditioner
$\mathcal{P}_{2,e}$ about the locking parameter $\lambda$ and mesh size $h$.

The implementation of the preconditioner $\mathcal{P}_{2,e}$ requires the exact inversion of
the leading block $A_1$. Recall that $A_1$ is the WG discretization of the Laplacian operator.
We use the conjugate gradient method with incomplete Cholesky decomposition as the preconditioner to solve
systems associated with $A_1$. The number of required iterations of GMRES with and without preconditioning
is listed in Table~\ref{P2e-PCG-A1}.
We can see that the incomplete Cholesky decomposition preconditioner is effective in reducing the number
of CG iterations required to reach convergence.
On the other hand, the number of iterations grows like $\mathcal{O}(h^{-1})$ for both
preconditioned and unpreconditioned systems as the mesh is refined.
This is consistent with observations made for
incomplete Cholesky decomposition preconditioners for second-order elliptic differential equations
(e.g., see \cite{Greenbaum_book}).

\begin{table}[tbh!]
\centering
\caption{The number of GMRES iterations required to reach convergence for preconditioned systems
for linear elasticity problem (\ref{elasticity-2})
with the preconditioner $\mathcal{P}_{2,e}$ \eqref{P2e} for rectangular meshes.}
\label{P2e-itr}
\begin{tabular}{|c|c|c|c|c|c|c|c|}
\hline
 \backslashbox{$\lambda$}{$1/h$}& \multicolumn{1}{c|}{8} & \multicolumn{1}{c|}{16} & \multicolumn{1}{c|}{32} & \multicolumn{1}{c|}{64} & \multicolumn{1}{c|}{128} & \multicolumn{1}{c|}{256} \\ \hline
1.4286e0 & 20 & 24 & 26 & 27 & 26 & 26 \\ 
1.6667e3 & 20 & 24 & 26 & 27 & 26 & 26 \\ 
1.6667e6 & 13 & 14 & 15 & 15 & 16 & 15 \\ \hline
\end{tabular}
\end{table}

\begin{table}[tbh!]
\centering
\caption{The number of Conjugate Gradient iterations required to reach convergence
for linear systems associated with $A_1$ (for vector Laplacian)
with and without incomplete Cholesky
preconditioning for rectangular meshes.}
\label{P2e-PCG-A1}
\begin{tabular}{|c|c|c|c|c|c|c|}
\hline
1/h & 8 & 16 & 32 & 64 & 128 & 256 \\ \hline
with preconditioning & 5 & 7 & 10 & 18 & 33 & 63 \\ \hline
without preconditioning & 81& 184& 371& 736& 1448&2831 
\\ \hline
\end{tabular}
\end{table}

\subsubsection{Linear poroelasticity}
\label{Section_LinPoro2by2}

This poroelasticity problem is taken from \cite{SISCLeePierMarRog2019} 
with the analytical solutions
\begin{align*}
    \V{u} = t 
    \begin{bmatrix}
        (-1 + \cos(2 \pi x)) \sin(2 \pi y)  + \frac{1}{\lambda + \mu} \sin(\pi x) \sin(\pi y) \\[0.08in]
        \sin(2 \pi x) (1 - \cos(2 \pi y)) + \frac{1}{\lambda + \mu} \sin(\pi x) \sin(\pi y)     
    \end{bmatrix},
    \quad
    p =  -t  \sin(\pi x) \sin(\pi y),
\end{align*}
and 
the right-hand side functions are
\begin{align*}
\V{f} & = 
 - t \begin{bmatrix}
        -8 \pi^2 \mu \cos(2 \pi x) \sin(2 \pi y) - \frac{2 \pi^2 \mu}{\lambda + \mu} \sin(\pi x) \sin(\pi y) \\
      + 4 \pi^2 \mu \sin(2\pi y) 
      + \pi^2 \cos(\pi x+\pi y)
      + \alpha \pi \cos(\pi x)\sin(\pi y) 
      \\[0.08in]
        8 \pi^2 \mu \sin(2 \pi x) \cos(2 \pi y) - \frac{2 \pi^2 \mu}{\lambda + \mu} \sin(\pi x) \sin(\pi y) \\
      - 4 \pi^2 \mu \sin(2\pi x) 
      + \pi^2 \cos(\pi x+\pi y)
      + \alpha \pi \sin(\pi x)\cos(\pi y)
\end{bmatrix},
\\ 
s & = -c_0 \sin(\pi x)\sin(\pi y) 
+ \frac{\pi \alpha}{\lambda + \mu} \sin(\pi x+\pi y)
	  +t \Big(2 \pi^2 \sin(\pi x) \sin(\pi y)\Big) .
\end{align*}
We take parameters $c_0 = 0$ and $1$, $E = 1$, $\kappa = 1$, $\lambda$ with several different values, $\mu$ being calculated using $E$ and $\lambda$, and $t = \Delta t$. 
This example is chosen to mimic a nearly incompressible system as $\lambda \rightarrow \infty$.

In Table \ref{P2_itr}, the number of GMRES iterations required to converge for the preconditioned system is listed for various values of
$h$, $\Delta t$, $c_0$, and $\lambda$.
For the preconditioned system $\mathcal{P}_{2}^{-1} A_{2}$ with $c_0 = 1$, we observe that the number of iterations keeps relatively constant with various $\lambda$ and $h$, especially for the locking cases $\lambda = 1.6667 \times 10^{3} \mbox{ and } 1.6667 \times 10^{6}$ where
GMRES takes only 2 steps to converge.
This confirms that $\mathcal{P}_{2}$ is robust about $\lambda$, $\Delta t$, and $h$.
When $\lambda$ is getting larger, the eigenvalues of the preconditioned system is more clustered around 1, so fewer iterations are needed to converge.
This observation matches the theoretical result that the eigenvalues of $\mathcal{P}_{2}^{-1} A_{2}$ are bounded by
\eqref{2-field-eigen}.
When $c_0 = 0$, the bound of the eigenvalues in \eqref{2-field-eigen} contains a factor ${\epsilon}/{\Delta t}$.
As predicted by this bound, the numerical results show that $\mathcal{P}_{2}$ is effective for all but
the case with $\Delta t \ll \epsilon$ (for instance, $\lambda = 1.4286$ and $c_0 = 0$ in the table).

To improve the computational efficiency of the preconditioner, we replace $D$ in $\mathcal{P}_{2}$ with its incomplete Cholesky decomposition.
As shown in Table \ref{P2_itr}, the number of GMRES iterations using $\mathcal{P}_{2,DLU}$ is almost the same as
that for $\mathcal{P}_{2}$ for $c_0 = 1$, especially for the locking cases.
For $c_0 = 0$, the performance of $\mathcal{P}_{2,DLU}$ is comparable with that of $\mathcal{P}_{2}$.

\begin{table}[tbh!]
\centering
\caption{The number of GMRES iterations required to reach convergence for preconditioned systems 
for linear poroelasticity problem (\ref{2by2Scheme_matrix2}) with
preconditioners $\mathcal{P}_{2}$ \eqref{P2} and $\mathcal{P}_{2,DLU} $ \eqref{P2DLU} for rectangular meshes, with $\kappa = 1$ and $\Delta t = 10^{-3}$ and $10^{-6}$.}
\label{P2_itr}
\begin{tabular}{|c|c|c|c|c|c||c|c|c|c|}
\hline
\multirow{3}{*}{$1/h$} & \multirow{3}{*}{$\lambda$} 
&\multicolumn{4}{c||}{$c_0 = 1$} & \multicolumn{4}{c|}{$c_0 = 0$} \\ \cline{3-10}
\multirow{2}{*}{} & \multirow{2}{*}{}&\multicolumn{2}{c|}{$\mathcal{P}_{2}$} 
& \multicolumn{2}{c||}
{$\mathcal{P}_{2,DLU}$} & \multicolumn{2}{c|}{$\mathcal{P}_{2}$} & \multicolumn{2}{c|}{$\mathcal{P}_{2,DLU}$} \\ \cline{3-10}
\multirow{2}{*}{} & \multirow{2}{*}{}& $10^{-3}$ &  $10^{-6} $ & $10^{-3} $ &  $10^{-6} $& $10^{-3} $ &  $10^{-6} $ & $10^{-3} $ &  $10^{-6}$
\\  \hline
\multirow{3}{*}{8} & 1.4286e0 & 7 & 6   & 7 & 8 &12  & 26& 14& 30\\ 
& 1.6667e3 & 2 & 2  &  2& 2 & 2 & 3 & 2&3\\ 
& 1.6667e6 &2  & 2  &2  & 2 & 2 & 2 & 2&2\\  
\hline \hline
\multirow{3}{*}{16} & 1.4286e0 & 7 & 7   & 7 & 7 &  13& 38& 16&59\\ 
& 1.6667e3 & 2 & 2  & 2 &2  &2  &3  & 2&3\\ 
& 1.6667e6 & 2 & 2  &2  &2  &2  &  2 &2 &2\\  
\hline \hline
\multirow{3}{*}{32} & 1.4286e0 & 7 & 7   & 7 & 7 & 13 & 87& 15& 112\\ 
& 1.6667e3 & 2 & 2 &  2& 2 & 2 & 3  & 2&2\\ 
& 1.6667e6 & 2 & 2  &2 & 2 &2  & 2 &2 &2 \\  
\hline \hline
\multirow{3}{*}{64} & 1.4286e0 & 7 & 7   & 7 & 7 & 13 &87 & 13& 134\\ 
& 1.6667e3 & 2 & 2  & 2 &  2& 2 & 3 &2 & 2\\ 
& 1.6667e6 & 2 & 2   & 2 & 2 & 2 & 2 & 2&2\\  
\hline \hline
\multirow{3}{*}{128} & 1.4286e0 & 7 & 7   & 6 & 7 & 13 &119 &9 &106\\ 
& 1.6667e3 & 2 &  2 & 2 & 2 & 2 & 3 & 2&2\\ 
& 1.6667e6 & 2 & 2  & 2 & 2 & 2 &2 &2 & 2\\ 
\hline \hline
\multirow{3}{*}{256} & 1.4286e0 & 7 &   7 & 9 & 7 & 13 & 126& 16&75 \\ 
& 1.6667e3 &2 & 2& 2 & 2 & 2 & 4 & 2&2\\ 
& 1.6667e6 & 2 &  2 & 2 & 2& 2 & 2 & 2&2\\  \hline 
\end{tabular}
\end{table}

\section{Three-field Schur complement preconditioning}
\label{SEC:3by3}

In this section we develop three-field Schur complement preconditioners for the iterative solution of (\ref{2by2Scheme_matrix2}).

\subsection{Three-field Schur complement preconditioners}
\label{SEC:three-field-schur}

In the previous section, we have seen that the two-field preconditioner (\ref{P2e}) is effective by transforming the linear elasticity problem into a new saddle-point system and using preconditioned GMRES.
A natural question is if we can use (\ref{w-1}) directly to transform (\ref{2by2Scheme_matrix2}) into a three-field
saddle-point system and obtain corresponding three-field preconditioners.

Like in Section~\ref{SEC:linear-elasticity}, we define a new variable $\mathbf{w}_h = - (M_p^{\circ})^{-1} B^{\circ} \mathbf{u}_h$
(cf. (\ref{w-1})). 
With \eqref{A0-2}, the first equation of the system \eqref{2by2Scheme_matrix2} can be rewritten as
\begin{align*}
    (\epsilon A_1 + A_0 )\V{u}_h + \frac{\alpha \epsilon}{\mu} B^T \V{p}_h = \epsilon A_1 \V{u}_h + \frac{\alpha \epsilon}{\mu} B^T \V{p}_h - (B^{\circ})^T \V{w}_h.
\end{align*}
Then,
we can rewrite \eqref{2by2Scheme_matrix2} into
\[
    \begin{bmatrix}
       \epsilon A_1  & \frac{\alpha \epsilon}{\mu} B^T & -(B^{\circ})^T\\[0.1in]
     \frac{\alpha \epsilon}{\mu} B & -\frac{\epsilon}{\mu} D & 0\\[0.1in]
     -B^{\circ} & 0& - M_{p}^{\circ}
    \end{bmatrix}
    \begin{bmatrix}
         \mathbf{u}_h \\[0.1in]
        \mathbf{p}_h \\[0.1in]
        \mathbf{w_h}
    \end{bmatrix}
    =
    \frac{\epsilon}{\mu} \begin{bmatrix}
        \mathbf{b}_1 \\[0.1in]
        \mathbf{b}_2 \\[0.1in]
       0
    \end{bmatrix} .
\]
By rescaling the unknown variables, we can further rewrite the above system into a more balanced form,
\begin{equation}
    \begin{bmatrix}
      A_1  &  B^T & -(B^{\circ})^T\\[0.1in]
     B & \displaystyle -\frac{\mu}{ \alpha^2 }D & 0\\[0.1in]
     -B^{\circ} & 0& \displaystyle -\epsilon M_{p}^{\circ}
    \end{bmatrix}
    \begin{bmatrix}
        \mathbf{u}_h \\[0.1in]
        \frac{\alpha}{\mu}\mathbf{p}_h \\[0.1in]
        \frac{1}{\epsilon} \mathbf{w_h}
    \end{bmatrix}
    =
     \begin{bmatrix}
       \frac{1}{\mu} \mathbf{b}_1 \\[0.1in]
       \frac{1}{\alpha} \mathbf{b}_2 \\[0.1in]
       0
    \end{bmatrix} .
    \label{3by3Scheme}
\end{equation}
In principle, we can work directly with the above system. Unfortunately, as for the two-field formulation (\ref{2by2Scheme_matrix2}),
this can lead to an undesired upper bound proportional to $1/\Delta t$ even for the case $c_0 > 0$
because part of the second block in (\ref{3by3Scheme}) becomes zero as $\Delta t \to 0$.
Once again, this is a distinct feature of the WG discretization. As in
Section~\ref{SEC:2-field-poroe}, we can eliminate unknown variable $\mathbf{p}_h^{\partial}$ from the system and get
\[
    \begin{bmatrix}
      A_1  &  (B^{\circ})^T & -(B^{\circ})^T\\[0.1in]
     B^{\circ} & \displaystyle -\frac{\mu}{ \alpha^2 }\tilde{D} & 0\\[0.1in]
     -B^{\circ} & 0& \displaystyle -\epsilon M_{p}^{\circ}
    \end{bmatrix}
    \begin{bmatrix}
        \mathbf{u}_h \\[0.1in]
        \frac{\alpha}{\mu}\mathbf{p}_h^{\circ} \\[0.1in]
        \frac{1}{\epsilon} \mathbf{w_h}
    \end{bmatrix}
    =
     \begin{bmatrix}
       \frac{1}{\mu} \mathbf{b}_1 \\[0.1in]
       \frac{1}{\alpha} \mathbf{b}_2^{\circ} \\[0.1in]
       0
    \end{bmatrix} ,
\]
where $\tilde{D}$ is defined in (\ref{D-3}). We can further formulate the above system into
\begin{align}
&    \tilde{\mathcal{A}}_{3}
    \begin{bmatrix}
        \mathbf{u}_h \\[0.1in]
        \frac{\alpha}{\mu}\mathbf{p}_h^{\circ} \\[0.1in]
        \frac{1}{\epsilon} \mathbf{w_h} - \frac{\alpha}{\mu}\mathbf{p}_h^{\circ}
    \end{bmatrix}
    =
     \begin{bmatrix}
       \frac{1}{\mu} \mathbf{b}_1 \\[0.1in]
       \frac{1}{\alpha} \mathbf{b}_2^{\circ} \\[0.1in]
       0
    \end{bmatrix} ,
\end{align}
where 
\begin{align}
    \tilde{\mathcal{A}}_{3} = 
\left [
\begin{array}{c|cc}
     A_1  &  0 & -(B^{\circ})^T\\
     \noalign{\vskip 0.1in} \hline \noalign{\vskip 0.1in}
     \displaystyle  0 &-\frac{\mu}{ \alpha^2 }\tilde{D}-\epsilon M_p^{\circ} & -\epsilon M_p^{\circ}\\[0.1in]
     \displaystyle  -B^{\circ} & -\epsilon M_p^{\circ} & -\epsilon M_{p}^{\circ} \\    
\end{array}
\right ]
\equiv
\begin{bmatrix}
    A_1 & \mathcal{B}^T \\[0.1in]
    \mathcal{B} & -\mathcal{D}
\end{bmatrix} .
    \label{3by3Scheme-2}
\end{align}
The Schur complement for this matrix is
\begin{align}
     \tilde{S}_3  = \mathcal{D} + \mathcal{B} A_1^{-1} \mathcal{B}^T =
         \begin{bmatrix}
           \displaystyle  \frac{\mu}{ \alpha^2 }\tilde{D}+\epsilon M_p^{\circ} & \epsilon M_p^{\circ}\\[0.1in]
     \displaystyle  \epsilon M_p^{\circ} & \epsilon M_{p}^{\circ} + B^{\circ} A_1^{-1} (B^{\circ})^T
       \end{bmatrix} .
\label{S3-1}
\end{align}

It is worth commenting that $\frac{1}{\epsilon} \mathbf{w_h} - \frac{\alpha}{\mu}\mathbf{p}_h^{\circ}$
corresponds to the total pressure $p_T = - \lambda \nabla \cdot \mathbf{u} + \alpha p$ 
that has been used by several researchers; e.g., see \cite{Boon4_SISC_2021,Lee-SISC-2017,Ricardo2_SINUM_2016}.

\begin{pro}
\label{pro:3-field}
\begin{align}
   \begin{bmatrix}
       \displaystyle  \frac{\mu}{ \alpha^2 }\tilde{D} & 0 \\[0.1in]
       0 & 0
   \end{bmatrix}
   \le \tilde{S}_3
   \le
  \begin{bmatrix}
   \displaystyle  2 \left (\frac{\mu}{ \alpha^2 }\tilde{D}+ \epsilon M_p^{\circ}\right )
   & 0 \\[0.1in]
   0 & M_p^{\circ} (2 \epsilon + { d})
\end{bmatrix} .
\label{pro:3-field-1}
\end{align}
\end{pro}

\begin{proof}
Consider the quadratic form of $\tilde{S}_3$. Using the Cauchy-Schwarz inequality we have
\begin{align*}
\begin{bmatrix}  \displaystyle
        \mathbf{u}^T & \mathbf{v}^T
    \end{bmatrix} \tilde{S}_3 \begin{bmatrix}
        \mathbf{u} \\  \displaystyle
        \mathbf{v}
    \end{bmatrix}
& =   \mathbf{u}^T \left ( \frac{\mu}{ \alpha^2 }\tilde{D}+\epsilon M_p^{\circ} \right ) \mathbf{u}
    + 2 \epsilon \mathbf{u}^T M_p^{\circ} \mathbf{v}
    + \mathbf{v}^T \left ( \epsilon M_{p}^{\circ} + B^{\circ} A_1^{-1} (B^{\circ})^T \right ) \mathbf{v} 
\\
& \leq \mathbf{u}^T \left ( \frac{\mu}{ \alpha^2 }\tilde{D}+ 2 \epsilon M_p^{\circ} \right ) \mathbf{u}
+ \mathbf{v}^T \left ( 2 \epsilon M_{p}^{\circ} + B^{\circ} A_1^{-1} (B^{\circ})^T \right ) \mathbf{v}
\\
& \le 2 \mathbf{u}^T \left ( \frac{\mu}{ \alpha^2 }\tilde{D}+ \epsilon M_p^{\circ} \right ) \mathbf{u} 
+ \mathbf{v}^T M_{p}^{\circ} \mathbf{v} ( 2 \epsilon + { d} ),
\end{align*}
where we have used (\ref{A0-2}) and (\ref{A0A1}).
Similarly, we have
\begin{align*}
\begin{bmatrix}  \displaystyle
        \mathbf{u}^T & \mathbf{v}^T
    \end{bmatrix} \tilde{S}_3 \begin{bmatrix}
        \mathbf{u} \\  \displaystyle
        \mathbf{v}
    \end{bmatrix}
& = \mathbf{u}^T \frac{\mu}{ \alpha^2 }\tilde{D} \mathbf{u}
    + \mathbf{v}^T B^{\circ} A_1^{-1} (B^{\circ})^T \mathbf{v} 
    + \epsilon (\mathbf{u} + \mathbf{v})^T M_p^{\circ} (\mathbf{u} + \mathbf{v}) 
\\
&
\ge \mathbf{u}^T \left ( \frac{\mu}{ \alpha^2 }\tilde{D} \right ) \mathbf{u}.
\end{align*}
The inequality (\ref{pro:3-field-1}) follows from the above results.
\end{proof}

Notice that $\tilde{S}_{3}$ becomes singular as $\epsilon \to 0$. Based on the discussion in Appendix~\ref{appendix_A},
we choose
\begin{equation}
\label{S3-2}
\hat{\tilde{S}}_3 = \begin{bmatrix}
   \displaystyle  \frac{\mu}{\alpha^2}\tilde{D} + \epsilon M_p^{\circ}
   & 0 \\[0.05in]
   0 & M_p^{\circ}
\end{bmatrix} 
\end{equation}
to take care of non-small eigenvalues of $\tilde{S}_3$. 
Thus, we define the three-field block upper triangular preconditioner as
\begin{align}
\tilde{\mathcal{P}}_{3} = 
    \begin{bmatrix}
          A_1  & 0 & -(B^{\circ})^T\\[0.05in]
        0 & \displaystyle -\frac{\mu}{\alpha^2}\tilde{D}-\epsilon M_p^{\circ} & 0 \\[0.05in]
        0 & 0 & - M_p^{\circ}
    \end{bmatrix} .
    \label{P3-0}
\end{align}
By Lemma~\ref{lem:SPP}, the eigenvalues of $\tilde{\mathcal{P}}_{3}^{-1} \tilde{A}_3$ consist of 1 and those
of $\hat{\tilde{S}}_3^{-1} \tilde{S}_3$.
From (\ref{S3-1}) and (\ref{S3-2})
it is not difficult to see that as $\epsilon \to 0$,
{the eigenvalues of $\hat{\tilde{S}}_3^{-1} \tilde{S}_3$}
are given by $1$ and the
eigenvalues of $(M_p^{\circ})^{-1} B^{\circ} A_1^{-1} (B^{\circ})^T$. As a result, the spectrum
of $\tilde{\mathcal{P}}_{3}^{-1} \tilde{A}_3$ contains
{an outlier of magnitude $\mathcal{O}(\epsilon)$ and
the cluster of eigenvalues in the interval $[\beta^2+\mathcal{O}(\epsilon),d+\mathcal{O}(\epsilon)]$.}
Like linear elasticity problems (cf. Section~\ref{SEC:linear-elasticity}), from the analysis of
Campbell et al. \cite{Campbell-1996} we can expect that 
the convergence factor of GMRES applied to the preconditioned system is bounded by the cluster radius
{$(d-\beta^2)/2$} (which is independent of $\Delta t$ and $h$) and thus, the preconditioner is robust
about the mesh size, time step, and locking parameter at least when $\epsilon$ is small.

As in Section~\ref{SEC:2-field-poroe}, we can transform the system and preconditioner back to the original
unknown variables. We have
\begin{equation}
    \mathcal{A}_3
    \begin{bmatrix}
        \mathbf{u}_h \\[0.05in]
        \frac{\alpha}{\mu}\mathbf{p}_h \\[0.05in]
        \frac{1}{\epsilon} \mathbf{w_h} - \frac{\alpha}{\mu}\mathbf{p}_h^{\circ}
    \end{bmatrix}
    =
     \begin{bmatrix}
       \frac{1}{\mu} \mathbf{b}_1 \\[0.05in]
       \frac{1}{\alpha} \mathbf{b}_2 \\[0.05in]
       0
    \end{bmatrix} ,
    \qquad 
    \mathcal{A}_3 = \begin{bmatrix}
      A_1  &  0 & -(B^{\circ})^T\\[0.05in]
     \displaystyle  0 &-\frac{\mu}{\alpha^2}\tilde{\tilde{D}}
     & -\epsilon \begin{pmatrix} M_p^{\circ}\\ 0 \end{pmatrix} \\[0.05in]
     \displaystyle  -B^{\circ} & -\epsilon \begin{pmatrix} M_p^{\circ} & 0 \end{pmatrix} & -\epsilon M_{p}^{\circ}
    \end{bmatrix} ,
    \label{3by3Scheme-3}
\end{equation}
\begin{align}
\label{P3}
&    \mathcal{P}_3 = \begin{bmatrix}
      A_1  &  0 & -(B^{\circ})^T\\[0.05in]
     \displaystyle  0 &-\frac{\mu}{ \alpha^2 }\tilde{\tilde{D}}
     & 0 \\[0.05in]
     \displaystyle  0 & 0 & - M_{p}^{\circ}
    \end{bmatrix},
\end{align}
where 
\begin{equation}
    \label{D-4}
    \tilde{\tilde{D}} = D+\frac{\alpha^2\epsilon}{\mu} \begin{pmatrix} M_p^{\circ} & 0\\ 0 & 0 \end{pmatrix}
    = \left (c_0 + \frac{\alpha^2\epsilon}{\mu}\right ) \begin{pmatrix} M_p^{\circ} & 0\\ 0 & 0 \end{pmatrix} + {  \kappa}\Delta t A_p .
\end{equation}
Note that the eigenvalues of $\mathcal{P}_3^{-1} \mathcal{A}_3$ consist of $1$ and those of $\tilde{\mathcal{P}}_3^{-1} \tilde{\mathcal{A}}_3$ and thus, as $\epsilon \to 0$, contain the outlier
{near} 0 and the cluster {approximately} in the interval $[\beta^2,{ d}]$.

To make this preconditioner more practical and efficient, we replace the (2,2) block with its incomplete Cholesky decomposition.
Thus, we have
\begin{align}
\mathcal{P}_{3,DLU} = 
    \begin{bmatrix}
      A_1  &  0 & -(B^{\circ})^T\\[0.1in]
     \displaystyle  0 &-\frac{\mu}{ \alpha^2 }L_{\tilde{\tilde{D}}} L_{\tilde{\tilde{D}}}^T
     & 0 \\[0.1in]
     \displaystyle  0 & 0 & - M_{p}^{\circ}
    \end{bmatrix} .
    \label{P3DLU}
\end{align}
Notice that the mass matrix $M_p^{\circ}$ is diagonal and there is no need to replace it by its Cholesky decomposition.

{We also apply the algebraic multigrid method (AMG) \cite{amg} (with two v-cycles) to approximate the 
(2,2) block. The resulting preconditioner is denoted by}
\begin{align}
\mathcal{P}_{3,DAMG} = 
    \begin{bmatrix}
      A_1  &  0 & -(B^{\circ})^T\\[0.1in]
     \displaystyle  0 & -\frac{\mu}{ \alpha^2} \text{AMG}({\tilde{\tilde{D}}})
     & 0 \\[0.1in]
     \displaystyle  0 & 0 & - M_{p}^{\circ}
    \end{bmatrix} .
    \label{P3DAMG}
\end{align}

\subsection{Numerical experiments for three-field preconditioning}

In this section we present numerical results to showcase the performances of
preconditioners $\mathcal{P}_{3}$,
{$\mathcal{P}_{3,DLU}$, and $\mathcal{P}_{3,DAMG}$}.

\subsubsection{Two-dimensional examples}
We test the same two-dimensional example as described in Section \ref{Section_LinPoro2by2}.
We first examine the performance of preconditioner $\mathcal{P}_{3}$ (\ref{P3}).
From the convergence analysis of GMRES in \cite{Campbell-1996}, 
we can expect GMRES to have a rapid, parameter-robust convergence for the preconditioned system.
Indeed, this can be verified by the results in Table~\ref{P3_itr} which show that
the number of GMRES iterations stays consistently bounded for various values of $\lambda$, $\Delta t$, and $h$
and $c_0 = 1$ and $c_0 = 0$,
which validates that $\mathcal{P}_{3}$ is effective and robust about $\lambda$, $\Delta t$, $h$, and $c_0$.
Particularly, $\mathcal{P}_{3}$ works well for $c_0 = 0$ and $\Delta t \ll \epsilon$
(for instance, $\lambda = 1.4286$, $c_0 = 0$ and $\Delta t = 10^{-3}$ and $10^{-6}$ in Table~\ref{P3_itr}),
a case when the two-field preconditioners $\mathcal{P}_2$ and $\mathcal{P}_{2,DLU}$ do not perform well. 

Now we examine the performance of $\mathcal{P}_{3,DLU}$ that is obtained by replacing $\tilde{\tilde{D}}$
with its incomplete Cholesky decomposition in $\mathcal{P}_{3}$.
Its performance is comparable with that of $\mathcal{P}_{3}$
for all but cases with $c_0 = 0$ and $h = 1/128$ and $1/256$; see Table~\ref{P3_itr}.
It seems that the accuracy in approximating $\tilde{\tilde{D}}$ plays a more significant role
in the convergence of GMRES when $c_0 = 0$ and the mesh is very fine.
Still, the number of GMRES iterations with $\mathcal{P}_{3,DLU}$ does not grow like $\mathcal{O}(h^{-1})$
as for second-order elliptic differential equations (see, e.g., \cite{Greenbaum_book} and Table~\ref{P2e-PCG-A1}). 
This may be an interesting research topic to investigate in near future.

At this point, we can claim that $\mathcal{P}_{3,DLU}$ is a good choice as the preconditioner for the iterative
solution of the linear poroelasticity system (\ref{3by3Scheme}). Like $\mathcal{P}_{3}$, it is effective and robust about
mesh size, time step, and locking parameter. Moreover, it requires only the exact inversion of
the leading block $A_1$. Recall that $A_1$ results from the WG discretization of the Laplacian operator. It can be solved
efficiently using multigrid or algebraic multigrid methods or Krylov subspace iterative methods preconditioned
with, for instance, incomplete Cholesky decomposition.
Results obtained with the algebraic multigrid method will be presented in the next example.

\begin{table}[tbh!]
\centering
\caption{The number of GMRES iterations required to reach convergence for preconditioned systems
for linear poroelasticity problem (\ref{3by3Scheme-3}) with
preconditioners $\mathcal{P}_{3}$ \eqref{P3} and $\mathcal{P}_{3,DLU} $ \eqref{P3DLU} for rectangular meshes, with $\kappa = 1$ and $\Delta t = 10^{-3}$ and $10^{-6}$.}
\label{P3_itr}
\begin{tabular}{|c|c|c|c|c|c||c|c|c|c|}
\hline
\multirow{3}{*}{$1/h$} & \multirow{3}{*}{$\lambda$} 
&\multicolumn{4}{c||}{$c_0 = 1$} & \multicolumn{4}{c|}{$c_0 = 0$} \\ \cline{3-10}
\multirow{2}{*}{} & \multirow{2}{*}{}
&\multicolumn{2}{c|}{$\mathcal{P}_{3}$}  
& \multicolumn{2}{c||} {$\mathcal{P}_{3,DLU}$}
& \multicolumn{2}{c|}{$\mathcal{P}_{3}$} 
& \multicolumn{2}{c|}{$\mathcal{P}_{3,DLU}$} \\ \cline{3-10}
\multirow{2}{*}{} & \multirow{2}{*}{}& $10^{-3}$ &  $10^{-6} $ & $10^{-3} $ &  $10^{-6} $&  $10^{-3} $ &  $10^{-6}$& $10^{-3}$ &  $10^{-6} $
\\  \hline
\multirow{3}{*}{8} & 1.4286e0 & 17 &17 & 17 & 18  &19 &26  & 20& 27\\ 
& 1.6667e3 & 19& 19  &  19& 19 & 19& 20&19 & 20\\ 
& 1.6667e6 &14 &14  & 14& 14 & 14& 14& 14&14 \\  
\hline \hline
\multirow{3}{*}{16} & 1.4286e0 &18  & 19&  18& 20  & 21 & 30& 21&31\\ 
& 1.6667e3 &25 & 25  &  25& 25  &25 & 26 & 25 & 26\\ 
& 1.6667e6 & 17 & 17& 17 &17  &17 & 17& 17 & 17 \\   
\hline \hline
\multirow{3}{*}{32} & 1.4286e0 & 19 & 20& 19 &20  & 22 &33 & 22& 34\\ 
& 1.6667e3 & 27&  27 & 27 & 27 & 27& 28& 27& 28\\ 
& 1.6667e6 & 18& 18 &  18 & 18 &18&18 &18 &18 \\    
\hline \hline
\multirow{3}{*}{64} & 1.4286e0 & 18 &  19 & 18 & 19 & 22 &34 &23 &34 \\ 
& 1.6667e3 & 28&  28 & 28 & 28 & 28& 29 & 32 & 30 \\ 
& 1.6667e6 & 18  &18 & 18 & 18 &18 & 18& 21&21\\    
\hline \hline
\multirow{3}{*}{128} & 1.4286e0 & 18 &19 &18  &19  &21 & 32 & 23&34\\ 
& 1.6667e3 & 29&  29 & 29 &29  & 29& 30 & 60& 31\\ 
& 1.6667e6 &18 & 18 & 19& 18 &  18& 18 & 25&  26\\  
\hline \hline
\multirow{3}{*}{256}& 1.4286e0 & 18 & 19&19  & 19  & 21& 31 &29 & 32\\ 
& 1.6667e3 & 29&  29 & 30 & 29 & 29& 31 & 196 & 50 \\ 
& 1.6667e6 & 18& 18 & 18& 18 & 18& 18& 29& 45\\   \hline 
\end{tabular}
\end{table}


%

\subsubsection{Three-dimensional examples}

In this subsection, we present some numerical results to showcase the effectiveness of the preconditioners in three dimensions.

We consider a three-dimensional example, which is a modification of the two-dimensional problem, with $\Omega = (0,1)\times (0,1)\times (0,1)$ and a cubic uniform mesh. 
We implement the preconditioners $\mathcal{P}_{3}$ and {$\mathcal{P}_{3,DAMG}$
on \texttt{deal.II} \cite{dealii}, where AMG is provided by \texttt{PETSc} \cite{petsc}}, and use the GMRES solver
\textit{SolverGMRES} with {relative tolerance $10^{-3}$, measured with respect to the norm of the right-hand side,} and
the default restart 28. 
The right-hand side functions are taken as
\begin{align*}
\mathbf{f} & = t
 \begin{bmatrix}
 4\mu \cos(2\pi x) \sin(2\pi y) \sin(2\pi z) \pi^2 + \frac{(4\mu + \lambda)}{(\mu + \lambda)} \sin(\pi x) \sin(\pi y) \sin(\pi z) \pi^2 \\
- \cos(\pi x) \cos(\pi y) \sin(\pi z) \pi^2 - \cos(\pi x) \sin(\pi y) \cos(\pi z) \pi^2 \\
+ 8\pi^2 \mu \left(-1 + \cos(2\pi x)\right) \sin(2\pi y) \sin(2\pi z) + \alpha \pi \cos(\pi x) \sin(\pi y) \sin(\pi z)
\\[0.08in]
- \pi^2 \cos(\pi x) \cos(\pi y) \sin(\pi z) + \frac{(4\mu + \lambda)}{(\mu + \lambda)} \sin(\pi x) \sin(\pi y) \sin(\pi z) \pi^2 \\
- \sin(\pi x) \cos(\pi y) \cos(\pi z) \pi^2 + 16\pi^2 \mu \sin(2\pi x) \left(1 - \cos(2\pi y)\right) \sin(2\pi z) \\
- 8\pi^2 \mu \sin(2\pi x) \cos(2\pi y) \sin(2\pi z) + \alpha \pi \sin(\pi x) \cos(\pi y) \sin(\pi z)
\\[0.08in]
 \frac{(4\mu + \lambda)}{(\mu + \lambda)} \sin(\pi x) \sin(\pi y) \sin(\pi z) \pi^2 + 4\pi^2 \mu \sin(2\pi x) \sin(2\pi y) \cos(2\pi z) \\
- \cos(\pi x) \sin(\pi y) \cos(\pi z) \pi^2 - \sin(\pi x) \cos(\pi y) \cos(\pi z) \pi^2 \\
+ 8\pi^2 \mu \left(-1 + \cos(2\pi z)\right) \sin(2\pi x) \sin(2\pi y) + \alpha \pi \sin(\pi x) \sin(\pi y) \cos(\pi z)
\end{bmatrix},
\\
s &= \frac{\alpha \pi}{\mu + \lambda} \Big( \cos(\pi x) \sin(\pi y) \sin(\pi z) + \sin(\pi x) \cos(\pi y) \sin(\pi z)
\\
& \qquad \qquad + \sin(\pi x) \sin(\pi y) \cos(\pi z) \Big)
 + (3 \pi^2 t + c_0)\sin(\pi x) \sin(\pi y) \sin(\pi z) .
\end{align*}
The other parameters are chosen the same as for two-dimensional examples.
The inversion of the (1,1) and (2,2) blocks in $\mathcal{P}_3$ and the inversion of the (1,1) block
in {$\mathcal{P}_{3,DAMG}$} are carried out through solving corresponding linear systems
using the conjugate gradient method preconditioned with {one v-cycle of AMG}.
{Moreover, two v-cycles of AMG 
are used to approximate $\tilde{\tilde{D}}$ in the (2,2) block for $\mathcal{P}_{3,DAMG}$.}

Table~\ref{P3_itr_3D} shows the number of GMRES iterations with the two preconditioners
for locking and non-locking cases, $\Delta t = 10^{-3}$ and
$10^{-6}$, and $c_0 = 1$ and $c_0 = 0$ on cubic meshes. We can see that the performance of
{$\mathcal{P}_{3,DAMG}$ is almost the same as that of $\mathcal{P}_3$ for all cases}. Moreover, they both work comparably well for $c_0 = 1$ and $c_0 = 0$.
Overall, the number of iterations stays relatively small and constant for the variations of the parameters,
which validates the effectiveness and robustness of the preconditioner
about the locking parameter, mesh size, and time step.

\begin{table}[h!]
\tiny
\centering
\caption{The number of GMRES iterations required to reach convergence for preconditioned systems
for linear poroelasticity problem (\ref{3by3Scheme-3}) with
preconditioners $\mathcal{P}_{3}$ \eqref{P3} and $\mathcal{P}_{3,DAMG} $ \eqref{P3DAMG} for cubic meshes, with $\kappa = 1$ and $\Delta t = 10^{-3}$ and $10^{-6}$.}
\label{P3_itr_3D}
\resizebox{\textwidth}{!}{
\begin{tabular}{|c|c|c|c|c|c||c|c|c|c|}
\hline
\multirow{3}{*}{$1/h$} & \multirow{3}{*}{$\lambda$} 
&\multicolumn{4}{c||}{$c_0 = 1$} & \multicolumn{4}{c|}{$c_0 = 0$} \\ \cline{3-10}
\multirow{2}{*}{} & \multirow{2}{*}{}
&\multicolumn{2}{c|}{$\mathcal{P}_{3}$}  
& \multicolumn{2}{c||} {$\mathcal{P}_{3,DAMG}$}
& \multicolumn{2}{c|}{$\mathcal{P}_{3}$} 
& \multicolumn{2}{c|}{$\mathcal{P}_{3,DAMG}$} \\ \cline{3-10}
\multirow{2}{*}{} & \multirow{2}{*}{}& $10^{-3}$ &  $10^{-6} $ & $10^{-3} $ &  $10^{-6} $&  $10^{-3} $ &  $10^{-6}$& $10^{-3}$ &  $10^{-6} $
\\  \hline
\multirow{3}{*}{8} 
& 1.4286e0 & 16 & 16  & 16& 16 & 21 & 19& 21& 19\\ 
& 1.6667e3 & 15 & 15 & 15 & 15&15 & 16& 15& 16 \\ 
& 1.6667e6 & 14 & 14  &14 & 14 & 14& 14& 14&14 \\  \hline \hline
\multirow{3}{*}{16}
& 1.4286e0 & 18 &  20   & 18& 20 & 25 & 47&26 &47 \\ 
& 1.6667e3 & 48 &   48   &  48& 48& 48& 48&48 &  48\\ 
& 1.6667e6 & 19 & 19 &19 & 19 & 19&19 & 19& 19\\  \hline  \hline
\multirow{3}{*}{32} 
& 1.4286e0 & 22 &   23  & 22&23  & 28 & 54&28 &52 \\ 
& 1.6667e3 & 52 &    52  & 52 & 52&52 &55 & 52&  55\\ 
& 1.6667e6 & 25 &   25  & 25& 25 & 25&25 &25& 25\\  \hline \hline 
\multirow{3}{*}{64} 
& 1.4286e0 & 24 &  27   &25 & 27 & 30 & 62& 32& 62\\ 
& 1.6667e3 & 57 & 57 & 57 & 57& 57&60 &57 & 59 \\ 
& 1.6667e6 & 29 &   29  &29 & 29 &29 & 29&29 &29 \\  \hline  
\end{tabular}
}
\end{table}

\section{Conclusions}
\label{SEC:conclusions}

In the previous sections, we have studied Schur complement preconditioning for efficient iterative solution of the linear algebraic system resulting from the implicit Euler-weak Galerkin discretization of linear poroelasticity problems. The system is a saddle-point matrix of two-by-two blocks (\ref{2by2Scheme_matrix2}) with the distinct features that the (2,1) and (1,2) blocks are rank deficient and the leading block is becoming singular for the locking situation with large $\lambda$ or small $\epsilon = \mu/(\lambda + \mu)$. These features make it non-trivial to develop efficient preconditioners using the current theory of preconditioning for general saddle-point problems. Instead of following the common practice to develop spectrally equivalent preconditioners, we have sought in this work preconditioners that can take care of non-small eigenvalues of the Schur complement while keeping small ones. This strategy has allowed us to explore and use upper bounds of the Schur complement and develop efficient parameter-robust preconditioners for linear poroelasticity and elasticity problems. 

More specifically, we have studied two-field and three-field block triangular preconditioners based on Schur complement preconditioning. The two-field preconditioners are given by $\mathcal{P}_{2}$ (\ref{P2}) and  $\mathcal{P}_{2,DLU}$ (\ref{P2DLU}). $\mathcal{P}_{2}$ is spectrally equivalent to the Schur complement and the eigenvalues of the corresponding preconditioned matrix are bounded by (\ref{2-field-eigen}). $\mathcal{P}_{2,DLU}$ is obtained by replacing the (2,2) block of $\mathcal{P}_{2}$ with its incomplete Cholesky decomposition. Both preconditioners are effective  for all ranges of the parameters except for the case where $c_0 = 0$ and ${ \kappa}\Delta t \ll \epsilon$. The implementation of these two-field preconditioners requires the solution of linear systems associated with the leading block
$A_0 + \epsilon A_1$, which are actually problems corresponding to linear elasticity. We have showed in Section~\ref{SEC:linear-elasticity} that a linear elasticity problem can be transformed into a saddle-point problem (with the coefficient matrix $\mathcal{A}_{2,e}$) that can be solved efficiently with the preconditioner $P_{2,e}$ (\ref{P2e}). Indeed, the eigenvalues of $P_{2,e}^{-1} \mathcal{A}_{2,e}$ are given in (\ref{PA2e-eigen}) and consist of $\epsilon$ and a cluster contained in the interval $[\epsilon+\beta^2, \epsilon + { d}]$, where $\beta$ (\ref{beta-1}) is a constant associated with the inf-sup condition for the weak Galerkin discretization of linear elasticity. It is known that GMRES can be used to solve this type of systems efficiently \cite{Campbell-1996}.

The two-field system (\ref{2by2Scheme_matrix2}) can be formulated into a three-field system (\ref{3by3Scheme-2}) with the coefficient matrix $\mathcal{A}_3$. The three-field upper triangular preconditioners are given by $\mathcal{P}_{3}$ (\ref{P3}), {$\mathcal{P}_{3,DLU}$  (\ref{P3DLU}), and $\mathcal{P}_{3,DAMG}$  (\ref{P3DAMG})}.
For the locking situation $\epsilon \to 0$, the eigenvalues of $\mathcal{P}_{3}^{-1} \mathcal{A}_3$ consist of {near} zero and a cluster contained in the interval $[\beta^2,{d}]$. Numerical results have confirmed that both $\mathcal{P}_{3}$,
{$\mathcal{P}_{3,DLU}$, and $\mathcal{P}_{3,DAMG}$} are effective and parameter-robust for the GMRES solution of (\ref{3by3Scheme-2}).

It is emphasized that all of the preconditioners considered in this work, $\mathcal{P}_{2}$, $\mathcal{P}_{2,DLU}$, $\mathcal{P}_{2,e}$, $\mathcal{P}_{3}$, {  $\mathcal{P}_{3,DLU}$, $\mathcal{P}_{3,DAMG}$}, do not require the computation of the Schur complement. 
Moreover, their implementation is straightforward. Particularly, the (2,2) and (3,3) blocks are either diagonal or can be replaced by an incomplete Cholesky decomposition {or algebraic multigrid} without significantly affecting the performance, as shown in the numerical experiments. Furthermore, they are robust about the locking parameter, mesh size, and time step and work for both $c_0 > 0$ and $c_0 = 0$.

We comment that the preconditioners considered in this work require the inversion of the leading block. 
For the linear elasticity preconditioner $\mathcal{P}_{2,e}$ and the poroelasticity preconditioners $P_{3}$,
{$P_{3,DLU}$, and $P_{3,DAMG}$},
the leading block is $A_1$, a discrete approximation of the Laplacian operator. The inversion of $A_1$ can be
done efficiently through solving related linear systems using iterative methods such as the multigrid, algebraic multigrid,
and Krylov subspace methods.
In our computation, we have used the conjugate gradient method preconditioned with {either} incomplete Cholesky decomposition {or AMG} for this purpose.
On the other hand, it is common to replace the leading block with approximations such as incomplete LU decompositions for further
improvements in efficiency; e.g., see \cite{BenziGolubLiesen-2005,Benzi2008}. We have done some primary numerical
investigations with incomplete Cholesky decomposition and found that it can significantly affect the performance of the preconditioners
especially for fine meshes. Finding more accurate replacements of the leading block for the preconditioners may deserve more investigations.

{Finally, we emphasize that while we have used the WG discretization for linear elasticity
and poroelasticity problems in this work, the approach we used to construct and analyze
block Shur complement preconditioners can also be applied to other stable and
locking-free discretization methods.}

\section*{Acknowledgments}
W.~Huang was supported in part by the Air Force Office of Scientific Research (AFOSR) grant FA9550-23-1-0571
and the Simons Foundation grant MPS-TSM-00002397.
{The authors are grateful to the anonymous referees for their valuable comments
and Wolfgang Bangerth for his help with \texttt{deal.II}.}


\appendix

\section{Schur complement preconditioning for general saddle-point problems}
\label{appendix_A}

We give a brief discussion on block upper triangular Schur complement preconditioning
and the eigenvalues of the preconditioned matrix for general saddle-point systems in this appendix.
The reader is referred to, e.g., \cite{BenziGolubLiesen-2005,Benzi2008} for complete discussion
of iterative solution and preconditioning for saddle-point systems.

We consider saddle-point problems in the general form
\begin{equation}
\label{SPP-1}
\mathcal{A} = \begin{bmatrix} A & B^T \\ C & - D\end{bmatrix} ,
\end{equation}
where $A$ is nonsingular, $B$ does not necessarily have full rank, and
the Schur complement $S = D + C A^{-1} B^T$ can be singular.
We consider the block triangular preconditioner
\begin{equation}
\label{SPP-2}
\mathcal{P}_{t} = \begin{bmatrix} A & B^T \\ 0 & - \hat{S} \end{bmatrix},
\end{equation}
where $\hat{S}$ is a nonsingular matrix whose choice will be discussed below.
The following lemma provides information about the distribution of the eigenvalues of
the preconditioned matrix $\mathcal{P}_{t}^{-1} \mathcal{A}$.
It is worth mentioning that such information is typically given in literature for ideal preconditioners
defined using the exact Schur complement; e.g., see \cite{MurphyGolubWathen_SISC_2000}. Here we consider approximations
of the Schur complement in (\ref{SPP-2}). One advantage of doing so is that the obtained information
about the distribution of the eigenvalues of the preconditioned matrix can be used
to choose $\hat{S}$.

\begin{lem}
\label{lem:SPP}
The eigenvalues of the preconditioned matrix, $\mathcal{P}_{t}^{-1} \mathcal{A}$, consist of $1$
and the eigenvalues of $\hat{S}^{-1} S$.
Moreover, $\mathcal{P}_{t}^{-1} \mathcal{A}$ satisfies $(\lambda-1)^2 p_{\hat{S}^{-1} S}^2(\lambda) = 0$,
where $p_{\hat{S}^{-1} S}(\lambda)$ denotes the minimal polynomial of $\hat{S}^{-1} S$.
Furthermore, for the case with $C = 0$ or $p_{\hat{S}^{-1} S}(\lambda) = \lambda - 1$,
$\mathcal{P}_{t}^{-1} \mathcal{A}$ satisfies $(\lambda-1) p_{\hat{S}^{-1} S}(\lambda) = 0$.
\end{lem}

\begin{proof}
The matrix $\mathcal{A}$ can be factorized as
\[
\mathcal{A} = \begin{bmatrix} I & 0 \\ C A^{-1} & I \end{bmatrix} 
\begin{bmatrix} A & 0 \\ 0 & -S \end{bmatrix}
\begin{bmatrix} I & A^{-1} B^T \\ 0 & I \end{bmatrix}
= \begin{bmatrix} A & 0 \\ C & -S \end{bmatrix} \begin{bmatrix} I & A^{-1} B^T \\ 0 & I \end{bmatrix} .
\]
Moreover, $\mathcal{P}_{t}^{-1}$ can be expressed as
\[
\mathcal{P}_{t}^{-1} = \begin{bmatrix} A^{-1} & A^{-1} B^T \hat{S}^{-1} \\ 0 & - \hat{S}^{-1} \end{bmatrix} .
\]
Using these, it is straightforward to verify that
\[
\mathcal{P}_{t}^{-1} \mathcal{A} = \begin{bmatrix} I & A^{-1} B^T \\ 0 & I \end{bmatrix}^{-1}
\begin{bmatrix} I & 0 \\ - \hat{S}^{-1} C  &  \hat{S}^{-1} S \end{bmatrix}
\begin{bmatrix} I & A^{-1} B^T \\ 0 & I \end{bmatrix} .
\]
Thus, $\mathcal{P}_{t}^{-1} \mathcal{A}$ is similar to the block lower triangular matrix 
\[
\mathcal{L} = \begin{bmatrix} I & 0 \\ - \hat{S}^{-1} C  &  \hat{S}^{-1} S \end{bmatrix}
\]
and its eigenvalues consist of $1$ and those of $\hat{S}^{-1} S$.

Since $\mathcal{P}_{t}^{-1} \mathcal{A}$ is similar to $\mathcal{L}$, we just need to show that $\mathcal{L}$ satisfies
the polynomials in question.
Notice that for any polynomial $q(\lambda)$, $q(\mathcal{L})$ is a block lower triangular matrix with diagonal blocks
$q(I)$ and $q(\hat{S}^{-1} S)$. Then, we have
\[
(\mathcal{L}-I) p_{\hat{S}^{-1} S}(\mathcal{L})
= \begin{bmatrix} 0 & 0 \\ * & \hat{S}^{-1} S - I \end{bmatrix} \; \begin{bmatrix} p_{\hat{S}^{-1} S}(I) & 0 \\ * & 0 \end{bmatrix}
= \begin{bmatrix} 0 & 0 \\ * & 0 \end{bmatrix} ,
\]
where $*$ stands for non-zero blocks.
Thus, $\left ( (\mathcal{L}-I) p_{\hat{S}^{-1} S}(\mathcal{L})\right )^2 = 0$.
That is, $\mathcal{L}$, and therefore, $\mathcal{P}_{t}^{-1} \mathcal{A}$, satisfy
$(\lambda-1)^2 p_{\hat{S}^{-1} S}^2(\lambda) = 0$.

For the case with $C = 0$, $\mathcal{L}$ is block diagonal. Then, 
\[
(\mathcal{L}-I) p_{\hat{S}^{-1} S}(\mathcal{L})
= \begin{bmatrix} 0 & 0 \\ 0 & \hat{S}^{-1} S - I \end{bmatrix} \; \begin{bmatrix} p_{\hat{S}^{-1} S}(I) & 0 \\ 0 & 0 \end{bmatrix}
= \begin{bmatrix} 0 & 0 \\ 0 & 0 \end{bmatrix} .
\]

Finally, for the case with $p_{\hat{S}^{-1} S}(\lambda) = \lambda - 1$, we have $\hat{S}^{-1} S = I$ and
\[
(\mathcal{L}-I) p_{\hat{S}^{-1} S}(\mathcal{L})
= \begin{bmatrix} 0 & 0 \\ * & 0 \end{bmatrix} \; \begin{bmatrix} 0 & 0 \\ * & 0 \end{bmatrix}
= \begin{bmatrix} 0 & 0 \\ 0 & 0 \end{bmatrix} .
\]
\end{proof}

The above lemma provides a useful guide on how to choose $\hat{S}$. For example, we can choose $\hat{S} = +S$
or $-S$ if $S$ is nonsingular.
The preconditioned matrix $\mathcal{P}_{t}^{-1} \mathcal{A}$ has a single eigenvalue $1$ for the former choice
and eigenvalues $\pm 1$ for the latter. Moreover, $\mathcal{P}_{t}^{-1} \mathcal{A}$ satisfies $(\lambda - 1)^2 = 0$
and $(\lambda - 1)^2 (\lambda + 1)^2 = 0$, respectively. These results are well known in the literature for the block triangular
preconditioner (\ref{SPP-2}); e.g., see \cite{MurphyGolubWathen_SISC_2000}.
Unfortunately, these choices of $\hat{S}$ are impractical in general since the matrix factor $C A^{-1} B^T$ involved in $S$
and linear systems associated with $S$ can be expensive to compute.
It is common to choose $\hat{S}$ to be a matrix spectrally equivalent to $S$; e.g.,
see \cite{BenziGolubLiesen-2005,Benzi2008}.
For this choice, the eigenvalues of $\mathcal{P}_{t}^{-1} \mathcal{A}$ lie in a finite interval around $1$
on the real axis of the complex plane and the preconditioned system can be solved efficiently by an iterative method.
However, constructing a replacement spectrally equivalent to $S$ is a difficult task in general.
This is especially so when $S$ is singular or almost singular.

In this work we consider a different strategy with which we do not seek to remove the singularity.
To explain this, we consider the choice
\begin{equation}
\label{SPP-3}
\hat{S} = (S^{+})^{+},
\end{equation}
where $S^{+}$ denotes the Moore-Penrose pseudo-inverse of $S$.
In terms of singular value decompositions, this can be expressed as 
\begin{align*}
& S = U \begin{bmatrix} \mu_1 & & & & & & 0 & 0\\ & \ddots & & & & & \vdots & \vdots \\  & & \mu_r & & & & 0 & 0\\
 & & & 0 & & & 0 & 0 \\ & & & & \ddots & & \vdots & \vdots \\ & & & & & 0 & 0 & 0\end{bmatrix} V^T,
\\
& \hat{S} = (S^{+})^{+} = U \begin{bmatrix} \mu_1 & & & & & & 0 & 0\\ & \ddots & & & & & \vdots & \vdots \\  & & \mu_r & & & & 0 & 0\\
 & & & 1 & & & 0 & 0 \\ & & & & \ddots & & \vdots & \vdots \\ & & & & & 1 & 0 & 0\end{bmatrix} V^T ,
\end{align*}
where $U$ and $V$ are orthogonal matrices. From this, it is not difficult to see that
the eigenvalues of $\hat{S}^{-1} S$, and therefore the eigenvalues of $\mathcal{P}_{t}^{-1} \mathcal{A}$,
consist of $1$ and $0$ when $S$ is singular. This implies that the preconditioner (\ref{SPP-3})
processes non-zero eigenvalues well while keeping the zero eigenvalues.
The resulting preconditioned system is rather ill-conditioned.
It is known that iterative methods such as GMRES can work well for linear systems of this type
(e.g., see \cite{BenziGolubLiesen-2005,Elman-2014}).
Note that (\ref{SPP-3}) is also impractical since the computation of the pseudo-inverse of a matrix of large size
can be costly if not prohibitive.
Nevertheless, the choice suggests that we can seek a replacement
of $S$ that can take care of ``non-small" eigenvalues of $\hat{S}^{-1} S$.
Motivated by this, we will explore bounds of $S$
and show that an upper bound of $S$ can often work for linear elasticity and poroelasticity problems.

For the symmetric case where $C=B$, $A$ is SPD, and $D$ is symmetric and positive semi-definite,
it is beneficial to choose $\hat{S}$ to be SPD. In this case, $S$ is symmetric and positive semi-definite.
The eigenvalues of $\hat{S}^{-1} S$, and therefore, those of $\mathcal{P}_{t}^{-1} \mathcal{A}$,
are real and non-negative.


\bibliographystyle{siamplain}

\end{document}